\documentclass[final,onefignum,onetabnum]{siamart220329}

\usepackage{lipsum}
\usepackage{amsfonts}
\usepackage{graphicx}
\usepackage{subfigure}

\usepackage{epstopdf}
\usepackage{algorithm}
\usepackage{algorithmicx}
\usepackage{algpseudocode}

\usepackage{enumerate}
\usepackage{booktabs}
\usepackage{threeparttable}
\usepackage{multirow}

\ifpdf
  \DeclareGraphicsExtensions{.eps,.pdf,.png,.jpg}
\else
  \DeclareGraphicsExtensions{.eps}
\fi

\newsiamremark{remark}{Remark}
\newsiamremark{hypothesis}{Hypothesis}
\crefname{hypothesis}{Hypothesis}{Hypotheses}
\newsiamthm{claim}{Claim}
\newsiamthm{example}{Example}

\headers{Newton-based alternating methods for ground states}{Pengfei Huang and Qingzhi Yang}

\title{Newton-based alternating methods for the ground state of a class of multi-component Bose-Einstein condensates\thanks{Submitted to the editors Date.
\funding{The second author was funded by the National Natural Science Foundation of China (No. 12071234) and a Key Program of Natural Science Foundation of Tianjin, China (No. 21JCZDJC00220).}}}

\author{Pengfei Huang\thanks{School of Mathematics, Hunan University, Changsha, Hunan 410082, China 
  (\email{huangpf@hnu.edu.cn}).}
\and Qingzhi Yang\thanks{Corresponding author. School of Mathematical Sciences and LPMC, Nankai University, Tianjin 300071, China
  (\email{qz-yang@nankai.edu.cn}).}
}

\usepackage{amsopn}

\ifpdf
\hypersetup{
  pdftitle={Newton-based alternating methods for the ground state of a class of multi-component Bose-Einstein condensates},
  pdfauthor={Pengfei Huang, and Qingzhi Yang}
}
\fi

\begin{document}

\maketitle

% REQUIRED
\begin{abstract}
	The computation of the ground state of special multi-component Bose-Einstein condensates (BECs) can be formulated as an energy functional minimization problem with spherical constraints. It leads to a nonconvex quartic-quadratic optimization problem after suitable discretizations. First, we generalize the Newton-based methods for single-component BECs to the alternating minimization scheme for multi-component BECs. Second, the global convergent alternating Newton-Noda iteration (ANNI) is proposed. In particular, we prove the positivity preserving property of ANNI  under mild conditions. Finally, our analysis is applied to a class of more general ``multi-block" optimization problems with spherical constraints. Numerical experiments are performed to evaluate the performance of proposed methods for different multi-component BECs, including pseudo spin-1/2, anti-ferromagnetic spin-1 and spin-2 BECs. These results support our theory and demonstrate the efficiency of our algorithms.
\end{abstract}

% REQUIRED

% REQUIRED

	\section{Introduction}
Since the first realization of Bose-Einstein condensates (BECs) was announced in 1995 \cite{anderson1995observation,davis1995bose,bradley1995evidence}, numerous researchers have been attracted into the theoretical studies and numerical methods for the single-component BEC \cite{andersen2004theory,dalfovo1999theory,bao2003ground,cances2010numerical,bao2013mathematical,antoine2017efficient}. While the pioneering experiments were conducted for single species of atoms, it is a natural generalization to explore the multi-component BEC system. Through optical confinements, various spinor condensates have been achieved and revealed exciting phenomena absent in single-component BECs, including pseudo spin-1/2, spin-1 and spin-2 condensates \cite{myatt1997production,stamper1998optical,barrett2001all,chang2004observation}. In this growing research direction, mathematical models and numerical simulation have been playing an important role in understanding the theoretical part of spinor BECs \cite{kawaguchi2012spinor,stamper2013spinor}. 

One of the fundamental problems in BECs is to find the ground state, which is defined as the minimizer of the Gross-Pitaevskii (GP) energy functional minimization problem subject to some physical constraints \cite{bao2018mathematical,bao2004ground}. The simplest multi-component BEC is the binary mixture. Furthermore, it has been proved that the energy functional minimization problem for spin-$F~(F=1,2)$ multi-component BECs can be reduced to a two-component BEC problem under some conditions \cite{chern2014kinetic,bao2018mathematical}. Computing the ground state of such a two-component BEC can be expressed as
\begin{equation}\label{equ:spin-1/2}
	\begin{aligned}
		\min~& E(\Phi) = \int_{\mathbb{R}^d}\left[\frac{1}{2}(|\nabla \phi_1|^2+|\nabla \phi_2|^2)+V(\mathbf{x})(|\phi_1|^2+|\phi_2|^2)\right.\\
		&\left.+\frac{1}{2}\beta_{11}|\phi_1|^4+\frac{1}{2}\beta_{22}|\phi_2|^4+\beta_{12}|\phi_1|^2|\phi_2|^2\right]d\mathbf{x}\\
		\text{s.t.}~&\|\phi_1\|^2 = \alpha,~\|\phi_2\|^2=1-\alpha,
	\end{aligned}
\end{equation}
where $\alpha\in(0,1)$. $\mathbf{x}\in\mathbb{R}^d~(d=1,2,3)$ is the spatial coordinate vector, $\Phi=(\phi_1(\mathbf{x}), \phi_2(\mathbf{x}))^T$ is a two-component vector wave function, $V(\mathbf{x})$ is a real-valued external trapping potential, $\beta_{11},\beta_{22},\beta_{12}\in\mathbb{R}$ are interaction constants \cite{kawaguchi2012spinor,bao2018mathematical}. 

With appropriate discretizations, many methods have been proposed for computing the ground state of multi-component BECs, which includes Gauss-Seidel-type methods for the vector GP equation \cite{chang2005gauss}, the normalized gradient flow method for spin-1 BECs \cite{bao2007mass,bao2013efficient,bao2008computing} and the projection gradient method for spin-2 BECs \cite{wang2014projection}. However, most of the existing numerical methods evolve from the gradient flow method, and thus converge slowly or even have no convergence guarantee in general. Recently, the state-of-art Riemannian optimization method solving minimization problems on matrix manifold \cite{absil2008optimization} has been introduced to compute the ground state of BECs. Tian et al. \cite{tian2020ground} proposed an efficient adaptive regularized Newton method for the spin-$F~(F=1,2,3,\cdots)$ cases with three different retractions on the manifold. On the other hand, there have been efficient Newton-based methods with locally quadratically convergence computing the ground state of single-component BECs \cite{liu2020positivity,du2022newton,huang2022newton}, which are easy to implement and require no complicate manifold optimization theory. Particularly, these methods are positivity presesrving for the BEC problem discretized with the finite difference scheme and thus can find a positive ground state, the existence and uniqueness of which has attracted much attention \cite{bao2013mathematical,bao2018mathematical,bao2013efficient}. To our best knowledge, the existing methods for the multi-component BECs cannot guarantee a positive solution unless the initial point is close enough to a positive ground state.

In this paper, we aim to provide an efficient global convergent algorithm for computing a class of multi-component BECs and explore the positivity preserving property of proposed algorithms under some conditions. In order to achieve this goal, we first consider the finite difference and pseudo-spectral discretization of the two-component BEC problem \eqref{equ:spin-1/2} with real-valued ground state, both of which lead to the following structured nonconvex optimization problem with spherical constraints
\begin{equation}\label{equ:problem}
	\begin{aligned}
		\underset{\mathbf{u},\mathbf{v}\in\mathbb{R}^n}{\min}~&f(\mathbf{u},\mathbf{v})=\frac{\beta_{11}}{2}\sum_{i=1}^nu_i^4+\mathbf{u}^TA_1\mathbf{u}+\frac{\beta_{22}}{2}\sum_{i=1}^nv_i^4+\mathbf{v}^TA_2\mathbf{v}+\beta_{12}\sum_{i=1}^nu_i^2v_i^2\\
		\text{s.t.}~&\mathbf{u}^T\mathbf{u}=1,~\mathbf{v}^T\mathbf{v}=1,
	\end{aligned}
\end{equation}
where $\beta_{11},~\beta_{22}>0$. $A_1$ and $A_2$ are the sum of discretization matrices for the negative Laplace operator and potential $V(\mathbf{x})$ with scaling. Throughout this paper, we assume they are real symmetric matrices. And when they are Hermitian matrices with the Fourier pseudo-spectral discretization, replacing $A_i~(i=1,2)$ with $(A_i+\bar{A_i})/2$ where $\bar{A_i} $ is the complex conjugate of $A_i$ for analysis. Its corresponding first-order optimality condition, the so-called
coupled nonlinear algebraic eigenvalue problem with eigenvector nonlinearity (CNEPv) is as follows
\begin{equation}\label{equ:CNEPv}
	\begin{aligned}
		\lambda\mathbf{u}=A_1\mathbf{u}+(\beta_{11}\text{diag}(\mathbf{u}^{[2]})+\beta_{12}\text{diag}(\mathbf{v}^{[2]}))\mathbf{u},~\mathbf{u}^T\mathbf{u}=1,\\
		\mu\mathbf{v}=A_2\mathbf{v}+(\beta_{12}\text{diag}(\mathbf{u}^{[2]})+\beta_{22}\text{diag}(\mathbf{v}^{[2]}))\mathbf{v},~\mathbf{v}^T\mathbf{v}=1,
	\end{aligned}
\end{equation}
where $(\lambda,\mathbf{u},\mu,\mathbf{v})$ is an eigenpair of CNEPv \eqref{equ:CNEPv}. The discretization of coupled GP equations can also lead to \eqref{equ:CNEPv} \cite{bao2004ground}. See the detailed transformation from \eqref{equ:spin-1/2} to \eqref{equ:problem} in section \ref{sec:numerical}. Through exploring the symmetric block structure of \eqref{equ:problem}, it is natural to fix $\mathbf{u}$ and $\mathbf{v}$ respectively, and solve the subproblem utilizing algorithms for the single-component BECs \cite{wu2017regularized,huang2022newton,du2022newton} alternatively. In particular, Newton-Noda iteration (NNI) has been successfully applied to solve nonnegative tensor eigenvalue problems \cite{liu2016positivity} and compute the positive ground states of nonlinear schr\"odinger equations \cite{du2022newton}. However, the convergence of NNI is heavily dependent on the so-called ``$M$-matrix" property \cite{du2022newton,huang2022newton}. To improve the efficiency, generalize the scope of application, and preserve the positivity under some conditions, we propose an alternating NNI (ANNI) based on the sufficient descent property of the objective function, which conducts only one-step modified NNI for $\mathbf{u}$ and $\mathbf{v}$ alternatively in each iteration. We show the global convergence of algorithms with no requirement for discretization schemes and ensure that ANNI has the positivity preserving property for \eqref{equ:problem} with the finite difference discretization.

The rest of this paper is organized as follows. We start with some preliminaries and notations in section \ref{sec:pre}. The alternating minimization scheme for \eqref{equ:problem} is introduced in section \ref{sec:ALM}. In section \ref{sec:ANNI}, we propose an alternating Newton-Noda iteration for solving the discretized optimization problem \eqref{equ:problem}. The convergence analysis and positivity preserving property are also provided. In section \ref{sec:general}, the extension of proposed algorithms and theoretical analysis to more general multi-block optimization problems is presented. The detailed numerical results are reported in section \ref{sec:numerical} to verify the theoretical results and performance of algorithms. Finally, concluding remarks are given in section \ref{sec:conclude}.

\section{Preliminaries and notations}\label{sec:pre}
Throughout this paper, we assume that $\beta_{11},~\beta_{22},~\beta_{12}>0$. A vector is denoted by the bold letter $\mathbf{u}$, and the capital letter $A$ denotes a matrix. $\|\cdot\|$ denotes the 2-norm of vectors and matrices. In addition, for a vector $\mathbf{u}=[u_1,u_2,\cdots,u_n]^T$,
\[\min(\mathbf{u})=\underset{i}\min~u_i,\quad\max(\mathbf{u})=\underset{i}{\max}~u_i.\]
We will also use $(\mathbf{u})_i$ to represent the $i$th component of $\mathbf{u}$. Let $\mathbf{u}^{[m]}=[u_1^m,u_2^m,\cdots u_n^m]^T$, and $\text{diag}(\mathbf{u})$ represents a diagonal matrix with the diagonal given by the vector $\mathbf{u}$. For $\mathbf{u}_1,\mathbf{u}_2\in\mathbb{R}^n$, $\mathbf{u}_1>(\ge)\mathbf{u}_2$ denotes that $(\mathbf{u}_1)_i>(\ge)(\mathbf{u}_2)_i$, $i=1,2,\cdots,n$. For a matrix $A$, $\lambda_{\min}(A)$ denotes the smallest eigenvalue of $A$.

\begin{definition}($M$-matrix \cite{varga1962iterative})
	A matrix $A\in\mathbb{R}^{n\times n}$ is called an $M$-matrix, if $A=sI-C$, where $C$ is nonnegative and $s\ge\rho(C)$. Here $\rho(C)$ is the spectral radius of $C$.
\end{definition}

\begin{definition}(Irreducibility/Reducibility\cite{varga1962iterative})
	A matrix $A\in\mathbb{R}^{n\times n}$ is called reducible, if there exists a nonempty proper index subset $I\subset\{1,2,\cdots,n\}$, such that
	\[a_{ij}=0,\quad\forall i\in I,~\forall j\notin I.\]
	If $A$ is not reducible, then we call $A$ irreducible.
\end{definition}

\begin{theorem}(\cite[Theorem 3.16, Corollary 3.21]{varga1962iterative})\label{thm:M-matrix}
	Let $A=sI-C$, where $C\in\mathbb{R}^{n\times n}$ is nonnegative. The following are equivalent:
	\begin{enumerate}[(i)]
		\item $A$ is a nonsingular $M$-matrix.
		\item $A^{-1}$ exists and $A^{-1}$ is nonnegative.
		\item There exists $\mathbf{x}>0$ such that $A\mathbf{x}>0$.
	\end{enumerate}
	Futhermore, if $A$ is a real, symmetric and nonsingular irreducible matrix, then $A^{-1}>0$ if and only if $A$ is positive definite.
\end{theorem}

\begin{lemma}(\cite[Theorem 7.7.3]{horn2012matrix})\label{lem:matrix analysis}
	Let $A$ and $B$ be $n$ by $n$ Hermitian matrices and suppose that $A$ is positive definite. If $B$ is positive semidefinite, then $A-B$ is positive semidefinite (respectively, $A-B$ is positive definite) if and only if $\rho(A^{-1}B)\le 1$ (respectively, $\rho(A^{-1}B)<1$).
\end{lemma}

Denote $f_{\mathbf{v}}(\mathbf{u})=f(\mathbf{u},\mathbf{v})$ for any fixed $\mathbf{v}$ and define $f_{\mathbf{u}}(\mathbf{v})$ for any fixed $\mathbf{u}$ similarly. For convenience, we simplify \eqref{equ:CNEPv} into the following form:
\begin{equation*}
	\begin{aligned}
		\mathcal{A}_{\mathbf{v}}(\mathbf{u})\mathbf{u}=\lambda\mathbf{u},~\mathbf{u}^T\mathbf{u}=1,\\
		\mathcal{A}_{\mathbf{u}}(\mathbf{v})\mathbf{v}=\mu\mathbf{v},~\mathbf{v}^T\mathbf{v}=1,
	\end{aligned}
\end{equation*}
where $\mathcal{A}_{\mathbf{v}}(\mathbf{u})=\beta_{11}\text{diag}(\mathbf{u}^{[2]})+(A_1+\beta_{12}\text{diag}(\mathbf{v}^{[2]}))$ and $\mathcal{A}_{\mathbf{u}}(\mathbf{v})$ is defined similarly.
Analogous to the notations for the NEPv \cite{du2022newton,huang2022newton}, we define
\begin{equation}\label{equ:F}
	F_{\mathbf{v}}(\mathbf{u},\lambda)=\begin{bmatrix}
		\mathbf{r}_{\mathbf{v}}(\mathbf{u},\lambda)\\\frac{1}{2}(1-\mathbf{u}^T\mathbf{u})
	\end{bmatrix},\quad \text{where }\mathbf{r}_\mathbf{v}(\mathbf{u},\lambda)=\mathcal{A}_{\mathbf{v}}(\mathbf{u})\mathbf{u}-\lambda\mathbf{u}.
\end{equation}
The Jacobian of \eqref{equ:F} is given by
\begin{equation*}
	F'_{\mathbf{v}}(\mathbf{u},\lambda)=
	\begin{bmatrix}
		J_{\mathbf{v}}(\mathbf{u},\lambda)&-\mathbf{u}\\
		-\mathbf{u}^T&0
	\end{bmatrix},
\end{equation*}
where $J_{\mathbf{v}}(\mathbf{u},\lambda)=3\beta_{11}\text{diag}(\mathbf{u}^{[2]})+(A_1+\beta_{12}\text{diag}(\mathbf{v}^{[2]}))-\lambda I$ is the derivative of $\mathbf{r}_{\mathbf{v}}(\mathbf{u},\lambda)$ with respect to $\mathbf{u}$. It is easy to obtain that $J_{\mathbf{v}}(\mathbf{u},\lambda)=\mathcal{A}_{\mathbf{v}}(\mathbf{u})-\lambda I+2\beta_{11}\text{diag}(\mathbf{u}^{[2]})$ and
\begin{equation}\label{equ:Jrelation}
	J_{\mathbf{v}}(\mathbf{u},\lambda)\mathbf{u}=\mathbf{r}_{\mathbf{v}}(\mathbf{u},\lambda)+2\beta_{11}\mathbf{u}^{[3]}.
\end{equation}
The counterpart definition with respect to $\mathbf{v}$ is parallel.

\section{The alternating minimization method}\label{sec:ALM}
In this section, we explore the block structure of \eqref{equ:problem}. It is natural to present the alternating minimization scheme for it as follows
\begin{equation}\label{equ:ALM}
	\begin{aligned}
		\mathbf{u}_{k+1}& =\underset{\|\mathbf{u}\|=1}{\text{arg}\min}~ \frac{\beta_{11}}{2}\sum_{i=1}^nu_i^4+\mathbf{u}^T(A_1+\beta_{12}\text{diag}(\mathbf{v}_{k}^{[2]}))\mathbf{u},\\
		\mathbf{v}_{k+1}& =\underset{\|\mathbf{v}\|=1}{\text{arg}\min}~ \frac{\beta_{22}}{2}\sum_{i=1}^nv_i^4+\mathbf{v}^T(A_2+\beta_{12}\text{diag}(\mathbf{u}_{k+1}^{[2]}))\mathbf{v}.
	\end{aligned}
\end{equation}
And it is obvious that each subproblem in \eqref{equ:ALM} indeed corresponds to a 	``single-component BEC" problem.
\begin{lemma}\label{lem:subproblem}
	If $\beta_{11},~\beta_{22}>0$, $A_1$ and $A_2$ are all irreducible nonsingular $M$-matrices, then each subproblem in the alternating minimization scheme \eqref{equ:ALM} has a unique positive global optimum.
\end{lemma}
\begin{proof}
	Due to the symmetry structure of \eqref{equ:problem}, we only need to prove it for the subproblem about $\mathbf{u}$. For any fixed $\mathbf{v}_k$, let 
	\[\tilde{f}_{\mathbf{v}_k}(\mathbf{u})=\frac{\beta_{11}}{2}\sum_{i=1}^nu_i^4+\mathbf{u}^T(A_1+\beta_{12}\text{diag}(\mathbf{v}_{k}^{[2]}))\mathbf{u}.\]
	According to \cite[Lemma 1, Remark 1]{huang2022finding}, if $\beta_{11}>0$,  $\min_{\|\mathbf{u}\|=1}\tilde{f}_{\mathbf{v}_k}(\mathbf{u})$ has a unique positive global optimum. 
\end{proof}

\begin{theorem}\label{thm:alm}
	If $\beta_{11},~\beta_{22}>0$, $A_1$ and $A_2$ are all irreducible nonsingular $M$-matrices, let $\{(\mathbf{u}_k,\mathbf{v}_k)\}$ be a sequence generated by the alternating minimization scheme \eqref{equ:ALM}. Then every limit point $(\mathbf{u}_*,\mathbf{v}_*)$ is a positive stationary point of \eqref{equ:problem}, i.e., there holds
	\[f(\mathbf{u}_*,\mathbf{v}_*)\le f(\mathbf{u},\mathbf{v}_*),~f(\mathbf{u}_*,\mathbf{v}_*)\le f(\mathbf{u}_*,\mathbf{v}),~\text{for any unit } \mathbf{u},~\mathbf{v}.\]
	Furthermore, if $\eqref{equ:CNEPv}$ has a unique positive eigenvector pair, the whole sequence converges to the global optimum of \eqref{equ:problem}.
\end{theorem}
\begin{proof}
	According to Lemma \ref{lem:subproblem}, the generated $\{(\mathbf{u}_k,\mathbf{v}_k)\}$ sequence is bounded and positive. Thus, there exists a convergence subsequence $\{(\mathbf{u}_{k_j},\mathbf{v}_{k_j})\}$ converging to nonnegative $(\mathbf{u}_*,\mathbf{v}_*)$.
	
	Since $f(\mathbf{u}_{k+1},\mathbf{v}_{k+1})\le f(\mathbf{u}_{k+1},\mathbf{v}_{k})\le f(\mathbf{u}_{k},\mathbf{v}_{k})$ and $f(\mathbf{u},\mathbf{v})$ is bounded below over the unit spheres, $\{f(\mathbf{u}_k,\mathbf{v}_k)\}$ converges to $f(\mathbf{u}_*,\mathbf{v}_*)$. 
	
	Set
	\[\hat{\mathbf{u}}=\underset{\|\mathbf{u}\|=1}{\text{arg}\min}~ f(\mathbf{u},\mathbf{v}_*),~ \hat{\mathbf{v}}=\underset{\|\mathbf{v}\|=1}{\text{arg}\min}~ f(\mathbf{u}_*,\mathbf{v}),\]
	then we have 
	\begin{equation*}
		\left\{
		\begin{aligned}
			f(\mathbf{u}_{k_j},\hat{\mathbf{v}})&\ge f(\mathbf{u}_{k_j},\mathbf{v}_{k_j}),\\ f(\hat{\mathbf{u}},\mathbf{v}_{k_j})&\ge f(\mathbf{u}_{k_j+1},\mathbf{v}_{k_j})\ge\cdots\ge f(\mathbf{u}_{k_{j+1}},\mathbf{v}_{k_{j+1}}).
		\end{aligned}
		\right.
	\end{equation*}
	Therefore, $f(\mathbf{u}_*,\hat{\mathbf{v}})\ge f(\mathbf{u}_*,\mathbf{v}_*)$, $f(\hat{\mathbf{u}},\mathbf{v}_*)\ge f(\mathbf{u}_*,\mathbf{v}_*)$. Thus $(\mathbf{u}_*,\mathbf{v}_*)$ is a stationary point of \eqref{equ:problem} by the definition of $\hat{\mathbf{u}},~\hat{\mathbf{v}}$.
	
	On the other hand, the nonnegative stationary point $(\mathbf{u}_*,\mathbf{v}_*)$ also satisfies \eqref{equ:CNEPv}, that is, there exist $\lambda$ and $\mu$ such that
	\begin{equation*}
		\begin{aligned}
			\lambda\mathbf{u}_*=A_1\mathbf{u}_*+(\beta_{11}\text{diag}(\mathbf{u}_*^{[2]})+\beta_{12}\text{diag}(\mathbf{v}_*^{[2]}))\mathbf{u}_*,~\mathbf{u}_*^T\mathbf{u}_*=1,\\
			\mu\mathbf{v}_*=A_2\mathbf{v}_*+(\beta_{12}\text{diag}(\mathbf{u}_*^{[2]})+\beta_{22}\text{diag}(\mathbf{v}_*^{[2]}))\mathbf{v}_*,~\mathbf{v}_*^T\mathbf{v}_*=1.
		\end{aligned}
	\end{equation*}
	We now prove that $\mathbf{u}_*,~\mathbf{v}_*>0$. If there exists a nonempty subset $I\subset\{1,2,\cdots,n\}$ such that $(\mathbf{u}_*)_i=0,~\forall i\in I$, then we have $\sum_{j\notin I}(A_1)_{ij}(\mathbf{u}_*)_j=0,~\forall i\in I$. This leads to that $(A_1)_{ij}=0,~\forall i\in I,~j\notin I$, which contradicts the irreducibility of $A_1$.
	
	If the nonnegative eigenpair of CNEPv\eqref{equ:CNEPv} is unique, then it is obvious that the whole sequence $\{(\mathbf{u}_k,\mathbf{v}_k)\}$ must converge to it. This positive eigenvector is also the global optimum of \eqref{equ:problem}, since $f(\mathbf{u},\mathbf{v})\ge f(|\mathbf{u}|,|\mathbf{v}|)$.
\end{proof}

\begin{remark}
	Bao et al. \cite{bao2018mathematical} proved that if $\beta_{11}, \beta_{22}>0$ and $\beta_{11}\beta_{22}-\beta_{12}^2\ge 0$, the two-component BEC energy functional minimization problem \eqref{equ:spin-1/2} has a unique nonnegative ground state.
\end{remark}

There have been many efficient algorithms for finding the positive global optimum for the ``single-component BEC" problem, as mentioned in the introduction. It is easy to implement \eqref{equ:ALM} with these algorithms such as Newton-Root-Finding itertion (NRI) \cite{huang2022newton} to solve the subproblems. When discretized by the finite difference scheme, Liu et al. \cite{liu2020positivity,du2022newton} also proposed an efficient quadratically convergent method, the Newton-Noda iteration (NNI), to compute the positive ground state of single-component BECs. At the $k$th iteration in the alternating minimization scheme \eqref{equ:ALM}, the Newton-Noda iteration (NNI) \cite{du2022newton,huang2022newton} for the finite difference discretized ``single-component BEC" subproblem with respect to $\mathbf{u}$ is given in Algorithm \ref{alg:NNI} as follows.
\begin{algorithm}[htbp]
	\caption{NNI \cite{du2022newton,huang2022newton}}
	\label{alg:NNI}
	\begin{algorithmic}[1]
		\State Given a feasible initial point $\mathbf{u}^0>0$ with $\|\mathbf{u}^0\|=1$.  $\lambda^0=\min(\frac{\mathcal{A}_{\mathbf{v}_k}(\mathbf{u}^0)\mathbf{u}^0}{\mathbf{u}^0})$.
		\For{$l=0,1,2,\cdots$}
		\State Solve the linear system $F'_{\mathbf{v}_k}(\mathbf{u}^l,\lambda^l)\begin{bmatrix}\Delta\mathbf{u}^l\\\delta^l\end{bmatrix}=-F_{\mathbf{v}_k}(\mathbf{u}^l,\lambda^l).$
		\State Let $\theta^l=1.$
		\State Compute $\mathbf{w}^{l+1}=\mathbf{u}^l+\theta^l\Delta\mathbf{u}^l$. Normalize the vector $\hat{\mathbf{u}}^{l+1}=\frac{\mathbf{w}^{l+1}}{\|\mathbf{w}^{l+1}\|}$.
		\State Compute $h_l(\theta^l)=\mathcal{A}_{\mathbf{v}_k}(\hat{\mathbf{u}}^{l+1})\hat{\mathbf{u}}^{l+1}-\lambda^l\hat{\mathbf{u}}^{l+1}$
		\While{$h_l(\theta^l)\not> 0$}
		\State $\theta^l=\frac{\theta^l}{2}$, go back to step 5.
		\EndWhile
		\State $\mathbf{u}^{l+1}=\hat{\mathbf{u}}^{l+1}$, compute $\lambda^{l+1}=\min(\frac{\mathcal{A}_{\mathbf{v}_k}(\mathbf{u}^{l+1})\mathbf{u}^{l+1}}{\mathbf{u}^{l+1}})$.
		\EndFor
		\State\Return $\mathbf{u}_{k+1}=\mathbf{u}^{l}$.
	\end{algorithmic}
\end{algorithm}

\section{Alternating Newton-Noda iteration}\label{sec:ANNI}
In \eqref{equ:ALM}, it may take too much time for solving subproblems. A natural idea for potential improvement is to solve each subproblem inexactly to some extent. For instance, one may wonder whether several steps or even only one step NNI of Algorithm \ref{alg:NNI} is enough, instead of solving the subproblem completely every time. On the other hand, NNI \cite{du2022newton} requires $A_1$ and $A_2$ to be $M$-matrices in order for the convergence solving subproblems in \eqref{equ:ALM}. However, within some scenarios like BEC problems, for accuracy reasons, pseudo-spectral approximation schemes are used. It leads to the constraints problem \eqref{equ:problem}, where $A_1$ and $A_2$ cease to be $M$-matrices. Therefore, NNI itself needs to be modifed for more general applications.

In this section, motivated by NNI, we present an alternating Newton-Noda iteration (ANNI), which only conduct one-step modified Newton-Noda step for $\mathbf{u}$ and $\mathbf{v}$ alternatively in each iteration. Meanwhile, it still guarantees sufficient descent of the objective value and leads to the global convergence even without the finite difference discretization.

In NNI, a key idea is to choose $\theta$ satisfying $h(\theta)>0$ as the step 7 of Algorithm \ref{alg:NNI}, so that the sequence $\{\lambda^l\}$ is strictly increasing and finally converges to an eigenvalue of the NEPv. However, it is unknown that whether a one-step NNI is enough, or even whether such a $\theta$ satisfying $h(\theta)>0$ still exists when $A_1$ and $A_2$ are not $M$-matrices. Motivated by this, we wonder whether an appreciate $\theta$ can be selected and lead to the sufficient descent of the objective value, instead of forcing $h(\theta)>0$. The ANNI for \eqref{equ:problem} is presented below in Algorithm \ref{alg:ANNI}. For each inner one-step modified Newton-Noda step, in addition to the different condition for choosing $\theta$, another difference with the NNI is that we also modify the strategy for selecting $\lambda$ to guarantee the convergence.
\begin{algorithm}[htb]
	\caption{Alternating Newton-Noda iteration (ANNI)}\label{alg:ANNI}
	\begin{algorithmic}[1]
		\State Given feasible initial points $\mathbf{u}_0,~\mathbf{v}_0>0$ with $\|\mathbf{u}_0\|=1,~\|\mathbf{v}_0\|=1$,  $\tau_1<\tau_2<\min\{\lambda_{\min}( \mathcal{A}_{\mathbf{v}}(\mathbf{u})),\lambda_{\min}( \mathcal{A}_{\mathbf{u}}(\mathbf{v}))\}$ for all unit $\mathbf{u},\mathbf{v}$.
		\For{$k=0,1,2,\cdots$}
		\State (One-step modified Newton-Noda iteration for $\mathbf{u}$, step 4.-step 11.)
		\State Select $\lambda_k\in[\tau_1,\tau_2]$.
		\State Solve the linear system $F'_{\mathbf{v}_k}(\mathbf{u}_k,\lambda_k)\begin{bmatrix}\Delta\mathbf{u}_k\\\delta_k\end{bmatrix}=-F_{\mathbf{v}_k}(\mathbf{u}_k,\lambda_k).$
		\State Let $\theta_k=1.$
		\State Compute $\mathbf{w}_{k+1}=\mathbf{u}_k+\theta_k\Delta\mathbf{u}_k$
		and  $d_k(\theta_k)=f_{\mathbf{v}_k}(\frac{\mathbf{w}_{k+1}}{\|\mathbf{w}_{k+1}\|})-f_{\mathbf{v}_k}(\mathbf{u}_k)$.
		%\While{$h_k(\theta_k)\not> 0 \text{ or } d_k(\theta_k)\ge 0$}
		\While{$d_k(\theta_k)\ge 0$}
		\State $\theta_k=\frac{\theta_k}{2}$, go back to step 7.
		\EndWhile
		\State $\mathbf{u}_{k+1}=\mathbf{w}_{k+1}/\|\mathbf{w}_{k+1}\|$.
		
		\State (One-step modifed Newton-Noda iteration for $\mathbf{v}$, step 13.-step 20.)
		\State Select $\mu_k\in[\tau_1,\tau_2]$.
		\State Solve the linear system $F'_{\mathbf{u}_{k+1}}(\mathbf{v}_k,\mu_k)\begin{bmatrix}\Delta\mathbf{v}_k\\\delta_k\end{bmatrix}=-F_{\mathbf{u}_{k+1}}(\mathbf{v}_k,\mu_k).$
		\State Let $\theta_k=1.$
		\State Compute $\mathbf{z}_{k+1}=\mathbf{v}_k+\theta_k\Delta\mathbf{v}_k$ and  $d_k(\theta_k)=f_{\mathbf{u}_{k+1}}(\frac{\mathbf{z}_{k+1}}{\|\mathbf{z}_{k+1}\|})-f_{\mathbf{u}_{k+1}}(\mathbf{v}_k)$.
		%\While{$h_k(\theta_k)\not> 0 \text{ or } d_k(\theta_k)\ge 0$}
		\While{$d_k(\theta_k)\ge 0$}
		\State $\theta_k=\frac{\theta_k}{2}$, go back to step 16.
		\EndWhile
		\State $\mathbf{v}_{k+1}=\mathbf{z}_{k+1}/\|\mathbf{z}_{k+1}\|$.
		\EndFor
	\end{algorithmic}
\end{algorithm}

\subsection{Properties of $\mathbf{u}_k$ and $\mathbf{v}_k$}
Suppose the sequence$\{(\mathbf{u}_k,\mathbf{v}_k,\lambda_k,\mu_k)\}$ is generated by Algorithm \ref{alg:ANNI}. Due to the symmetry of $\mathbf{u}$ and $\mathbf{v}$ in \eqref{equ:problem}, we only need to derive the property on $\mathbf{u}$ without statement otherwise, and the result with respect to $\mathbf{v}$ follows. First, let us explore some basic properties of the linear system
\begin{equation}\label{equ:linear}
	F'_{\mathbf{v}}(\mathbf{u},\lambda)\begin{bmatrix}\Delta\mathbf{u}\\\delta\end{bmatrix}=\begin{bmatrix}-\mathbf{r}_{\mathbf{v}}(\mathbf{u},\lambda)\\0\end{bmatrix}.
\end{equation}
It is obvious to obtain that
\begin{equation}\label{equ:relation}
	\begin{aligned}	J_{\mathbf{v}}(\mathbf{u},\lambda)\Delta\mathbf{u}&=\delta\mathbf{u}-\mathbf{r}_{\mathbf{v}}(\mathbf{u},\lambda),\\
		\mathbf{u}^{T}\Delta\mathbf{u}&=0,\\
		J_{\mathbf{v}}(\mathbf{u},\lambda)(\mathbf{u}+\Delta\mathbf{u})&=\delta\mathbf{u}+2\beta_{11}\mathbf{u}^{[3]}.
	\end{aligned}
\end{equation}

\begin{lemma}\label{lem:deltabound}
	There exists positive constant $\gamma$, such that
	$\lambda_{\min}(J_{\mathbf{v}_k}(\mathbf{u}_k,\lambda_k))\ge \gamma$ for any $k$. 
	Suppose $\delta_k$ and $\Delta\mathbf{u}_k$ are generated by Algorithm \ref{alg:ANNI}, then 
	\begin{equation}\label{equ:delta2}
		\delta_k=\frac{\mathbf{u}_k^TJ_{\mathbf{v}_k}^{-1}(\mathbf{u}_k,\lambda_k)\mathbf{r}_{\mathbf{v}_k}(\mathbf{u}_k,\lambda_k)}{\mathbf{u}_k^TJ_{\mathbf{v}_k}^{-1}(\mathbf{u}_k,\lambda_k)\mathbf{u}_k},
	\end{equation}
	and 
	\begin{equation}\label{equ:Delta2}
		\Delta\mathbf{u}_k = \delta_kJ_{\mathbf{v}_k}^{-1}(\mathbf{u}_k,\lambda_k)\mathbf{u}_k-J_{\mathbf{v}_k}^{-1}(\mathbf{u}_k,\lambda_k)\mathbf{r}_{\mathbf{v}_k}(\mathbf{u}_k,\lambda_k)
	\end{equation}
	are bounded. 
\end{lemma}
\begin{proof}
	Since $\lambda_k\le\tau_2<\lambda_{\min}(\mathcal{A}_{\mathbf{v}_k}(\mathbf{u}_k))$ for any $k$, $J_{\mathbf{v}_k}(\mathbf{u}_k,\lambda_k)=\mathcal{A}_{\mathbf{v}_k}(\mathbf{u}_k)-\lambda_k I+2\beta_{11}\text{diag}(\mathbf{u}_k^{[2]})$ is always positive semidifinite and $\lambda_{\min}(J_{\mathbf{v}_k}(\mathbf{u}_k,\lambda_k))\ge \gamma>0$ for some $\gamma$. Then
	\[F_{\mathbf{v}}'(\mathbf{u},\lambda)=\begin{bmatrix}I&0\\-\mathbf{u}^T(J_{\mathbf{v}}(\mathbf{u},\lambda))^{-1}&1\end{bmatrix}\begin{bmatrix}J_{\mathbf{v}}(\mathbf{u},\lambda)&-\mathbf{u}\\0&-\mathbf{u}^T(J_{\mathbf{v}}(\mathbf{u},\lambda))^{-1}\mathbf{u}\end{bmatrix}\]
	is nonsigular and 
	\begin{equation*}\label{equ:inverse}
		(F'_{\mathbf{v}}(\mathbf{u},\lambda))^{-1}=
		\begin{bmatrix}
			(J_{\mathbf{v}}(\mathbf{u},\lambda))^{-1}-\frac{(J_{\mathbf{v}}(\mathbf{u},\lambda))^{-1}\mathbf{u}\mathbf{u}^T(J_{\mathbf{v}}(\mathbf{u},\lambda))^{-1}}{\mathbf{u}^T(J_{\mathbf{v}}(\mathbf{u},\lambda))^{-1}\mathbf{u}} & -\frac{(J_{\mathbf{v}}(\mathbf{u},\lambda))^{-1}\mathbf{u}}{\mathbf{u}^T(J_{\mathbf{v}}(\mathbf{u},\lambda))^{-1}\mathbf{u}}\\
			-\frac{\mathbf{u}^T(J_{\mathbf{v}}(\mathbf{u},\lambda))^{-1}}{\mathbf{u}^T(J_{\mathbf{v}}(\mathbf{u},\lambda))^{-1}\mathbf{u}} & -\frac{1}{\mathbf{u}^T(J_{\mathbf{v}}(\mathbf{u},\lambda))^{-1}\mathbf{u}}
		\end{bmatrix}.
	\end{equation*}
	Combined with \eqref{equ:relation}, we have 
	\begin{equation*}\label{equ:delta}
		\delta_k=\frac{\mathbf{u}_k^TJ_{\mathbf{v}_k}^{-1}(\mathbf{u}_k,\lambda_k)\mathbf{r}_{\mathbf{v}_k}(\mathbf{u}_k,\lambda_k)}{\mathbf{u}_k^TJ_{\mathbf{v}_k}^{-1}(\mathbf{u}_k,\lambda_k)\mathbf{u}_k},
	\end{equation*} 
	\begin{equation*}\label{equ:Delta}
		\Delta\mathbf{u}_k = \delta_kJ_{\mathbf{v}_k}^{-1}(\mathbf{u}_k,\lambda_k)\mathbf{u}_k-J_{\mathbf{v}_k}^{-1}(\mathbf{u}_k,\lambda_k)\mathbf{r}_{\mathbf{v}_k}(\mathbf{u}_k,\lambda_k).
	\end{equation*}
	The boundedness of $\delta_k$ and $\Delta\mathbf{u}_k$ then follows.
\end{proof}

\begin{lemma}\label{lem:delta}
	Suppose $\delta_k$ is generated by Algorithm \ref{alg:ANNI}, and furthermore
	\begin{equation} \label{equ:tau_for_positive}
		\tau_2<\min\{\lambda_{\min}( \mathcal{A}_{\mathbf{v}}(\mathbf{u})),\lambda_{\min}( \mathcal{A}_{\mathbf{u}}(\mathbf{v}))\}-2\max\{\beta_{11},\beta_{22}\}
	\end{equation}
 for any unit $\mathbf{u},~\mathbf{v}$, then $\delta_k>0$.
\end{lemma}
\begin{proof}
	According to \eqref{equ:Jrelation}, we have that
	\begin{equation*}
		\begin{aligned}
			\mathbf{u}_k^TJ_{\mathbf{v}_k}^{-1}(\mathbf{u}_k,\lambda_k)\mathbf{r}_{\mathbf{v}_k}(\mathbf{u}_k,\lambda_k)&=\mathbf{u}_k^TJ_{\mathbf{v}_k}^{-1}(\mathbf{u}_k,\lambda_k)(J_{\mathbf{v}_k}(\mathbf{u}_k,\lambda_k)\mathbf{u}_k-2\beta_{11}\mathbf{u}_k^{[3]})\\
			&=1-\mathbf{u}_k^TJ_{\mathbf{v}_k}^{-1}(\mathbf{u}_k,\lambda_k)2\beta_{11}\text{diag}(\mathbf{u}_k^{[2]})\mathbf{u}_k.\\
		\end{aligned}
	\end{equation*}
Combined with \eqref{equ:tau_for_positive}, we have 
\begin{equation*}
	\begin{aligned}
	\lambda_{\min}(J_{\mathbf{v}_k}(\mathbf{u}_k,\lambda_k))&=\lambda_{\min}\left(\mathcal{A}_{\mathbf{v}_k}(\mathbf{u}_k)-\lambda_kI+2\beta_{11}\text{diag}(\mathbf{u}_k^{[2]})\right)\\
	&\ge \lambda_{\min}(\mathcal{A}_{\mathbf{v}_k}(\mathbf{u}_k))-\lambda_k\ge \lambda_{\min}(\mathcal{A}_{\mathbf{v}_k}(\mathbf{u}_k))-\tau_2\\
	&> 2\beta_{11}.
	\end{aligned}
\end{equation*}
Then $\mathbf{u}_k^TJ_{\mathbf{v}_k}^{-1}(\mathbf{u}_k,\lambda_k)\mathbf{r}_{\mathbf{v}_k}(\mathbf{u}_k,\lambda_k)>0$. From \eqref{equ:delta2}, we get $\delta_k>0$.
\end{proof}

\begin{remark}
	\eqref{equ:tau_for_positive} is only a sufficient condition for making $\delta_k\ge 0$. For instance, from the proof of Lemma \ref{lem:M of J}, if $J_{\mathbf{v}_k}(\mathbf{u}_k,\lambda_k)$ is a nonsingular $M$-matrix, $\delta_k\ge 0$ with $\lambda_k=\min\left(\mathcal{A}_{\mathbf{v}_k}(\mathbf{u}_k)\mathbf{u}_k/\mathbf{u}_k\right)$. If $J_{\mathbf{v}_k}^{-1}(\mathbf{u}_k,\lambda_k)2\beta_{11}\text{diag}(\mathbf{u}_k^{[2]})$ is still a symmetric matrix, then $\delta_k$ is also positive from Lemma \ref{lem:matrix analysis} without \eqref{equ:tau_for_positive}.
\end{remark}

Through \eqref{equ:relation} and the positive definiteness of $J_{\mathbf{v}_k}(\mathbf{u}_k,\lambda_k)$, it is obvious to obtain the following lemma.
\begin{lemma}\label{lem:urelation}
	$\Delta\mathbf{u}_k=0$ if and only if $\mathbf{u}_k$ is an eigenvector of $\mathcal{A}_{\mathbf{v}_k}(\mathbf{u_k})$.
\end{lemma}

For the two-component BEC problem \eqref{equ:problem} discretized by the finite difference method, both $A_1$ and $A_2$ are irreducible nonsingular $M$-matrices. In this case, ANNI further preserves the positivity of $\mathbf{u}_k$, $\mathbf{v}_k$.
\begin{theorem}(The positivity preserving property)\label{thm:positivity}
	If both $A_1$ and $A_2$ are irreducible nonsingular $M$-matrices and $\tau_2$ satisfies \eqref{equ:tau_for_positive}, then for any $\theta_k\in(0,1]$, $\mathbf{u}_k,~\mathbf{v}_k>0$.
\end{theorem}
\begin{proof}
	From \eqref{equ:relation}, one has
	\[\mathbf{u}_k+\Delta\mathbf{u}_k=J^{-1}_{\mathbf{v}_k}(\mathbf{u}_k,\lambda_k)(\delta_k\mathbf{u}_k+2\beta_{11}\mathbf{u}_k^{[3]}).\]
	Since $A_1$ is an irreducible nonsingular M-matrix, $J_{\mathbf{v}_k}(\mathbf{u}_k,\lambda_k)$ is also an irreducible positive definite $M$-matrix, and $J^{-1}_{\mathbf{v}_k}(\mathbf{u}_k,\lambda_k)>0$ according to Theorem \ref{thm:M-matrix}. Since $\mathbf{u}_0>0$ and $\delta_k>0$ due to Lemma \ref{lem:delta}, we have that $\mathbf{u}_k+\Delta\mathbf{u}_k>0$. The desired result follows.
\end{proof}

\subsection{Properties of $\theta_k$}
Between the NNI for NEPv discretized via the finite difference scheme \cite{du2022newton,huang2022newton} and the one-step modified Newton-Noda iteration in ANNI for CNEPv\eqref{equ:CNEPv}, one of the main differences is that we require a $\theta_k$ satisfying $d_k(\theta_k)<0$. As stated before, our purpose is mainly to find an appreciate $\theta$ to guarantee the sufficient descent of the objective value in \eqref{equ:problem}. In this section, we will prove that Algorithm \ref{alg:ANNI} can always find such a $\theta_k$ which is bounded below by some positive constant within finite halving steps.
\begin{lemma}\label{lem:measure_d}
	Let $d_k(\theta_k)=f_{\mathbf{v}_k}(\mathbf{u}_{k+1})-f_{\mathbf{v}_k}(\mathbf{u}_{k})$, then
	\begin{equation}\label{equ:measure_d}
		d_k(\theta_k)=2\left(\frac{1}{\|\mathbf{w}_{k+1}\|}-1\right)\mathbf{r}_{\mathbf{v}_k}(\mathbf{u}_{k},\lambda_k)^T\mathbf{u}_{k}+\frac{2\theta_k}{\|\mathbf{w}_{k+1}\|}\mathbf{r}_{\mathbf{v}_k}(\mathbf{u}_{k},\lambda_k)^T\Delta\mathbf{u}_k+R(\theta_k\Delta\mathbf{u}_k),
	\end{equation}
	where $|R(\theta_k\Delta\mathbf{u}_k)|\le M\|\theta_k\Delta\mathbf{u}_k\|^2$ for some positive constant M only determined by $\beta_{11},~\beta_{12}$ and $A_1$.
\end{lemma}
\begin{proof}
	We have that
	\begin{equation*}
		\begin{aligned}
			&d_k(\theta_k)=f_{\mathbf{v}_k}(\mathbf{u}_{k+1})-f_{\mathbf{v}_k}(\mathbf{u}_{k})\\
			=&\left[(\mathcal{A}_{\mathbf{v}_{k}}(\mathbf{u}_{k+1})\mathbf{u}_{k+1})^T\mathbf{u}_{k+1}-(\mathcal{A}_{\mathbf{v}_{k}}(\mathbf{u}_{k})\mathbf{u}_{k})^T\mathbf{u}_{k}\right]-\left[\frac{\beta_{11}}{2}\sum_{i=1}^n(\mathbf{u}_{k+1})_i^4-\frac{\beta_{11}}{2}\sum_{i=1}^n(\mathbf{u}_{k})_i^4\right]\\
			=&\left[\mathbf{r}_{\mathbf{v}_k}(\mathbf{u}_{k+1},\lambda_k)^T\mathbf{u}_{k+1}-\mathbf{r}_{\mathbf{v}_k}(\mathbf{u}_{k},\lambda_k)^T\mathbf{u}_{k}\right]
			-\left[\frac{\beta_{11}}{2}\sum_{i=1}^n(\mathbf{u}_{k+1})_i^4-\frac{\beta_{11}}{2}\sum_{i=1}^n(\mathbf{u}_{k})_i^4\right]\\
			=&[\mathbf{r}_{\mathbf{v}_k}(\mathbf{u}_{k},\lambda_k)+J_{\mathbf{v}_k}(\mathbf{u}_{k},\lambda_k)(\mathbf{u}_{k+1}-\mathbf{u}_k)+E_k]^T\mathbf{u}_{k+1}\\
			&-\mathbf{r}_{\mathbf{v}_k}(\mathbf{u}_{k},\lambda_k)^T\mathbf{u}_{k}-\left[(2\beta_{11}\mathbf{u}_k^{[3]})^T(\mathbf{u}_{k+1}-\mathbf{u}_k)+\bar{E}_k\right]\\
			=&\mathbf{r}_{\mathbf{v}_k}(\mathbf{u}_{k},\lambda_k)^T\left(\frac{\mathbf{u}_k+\theta_k\Delta\mathbf{u}_k}{\|\mathbf{w}_{k+1}\|}\right)-\mathbf{r}_{\mathbf{v}_k}(\mathbf{u}_{k},\lambda_k)^T\mathbf{u}_{k}+\mathbf{u}_{k+1}^TJ_{\mathbf{v}_k}(\mathbf{u}_{k},\lambda_k)(\mathbf{u}_{k+1}-\mathbf{u}_k)\\
			&+E_k^T\mathbf{u}_{k+1}-\left[(J_{\mathbf{v}_k}(\mathbf{u}_{k},\lambda_k)\mathbf{u}_k-\mathbf{r}_{\mathbf{v}_k}(\mathbf{u}_{k},\lambda_k))^T(\mathbf{u}_{k+1}-\mathbf{u}_k)+\bar{E}_k\right]\\
			=&2\left(\frac{1}{\|\mathbf{w}_{k+1}\|}-1\right)\mathbf{r}_{\mathbf{v}_k}(\mathbf{u}_{k},\lambda_k)^T\mathbf{u}_{k}+\frac{2\theta_k}{\|\mathbf{w}_{k+1}\|}\mathbf{r}_{\mathbf{v}_k}(\mathbf{u}_{k},\lambda_k)^T\Delta\mathbf{u}_k\\
			&+(\mathbf{u}_{k+1}-\mathbf{u}_{k})^TJ_{\mathbf{v}_k}(\mathbf{u}_{k},\lambda_k)(\mathbf{u}_{k+1}-\mathbf{u}_{k})+E_k^T\mathbf{u}_{k+1}-\bar{E}_k, 
		\end{aligned}
	\end{equation*}
	where the second equality comes from that $\|\mathbf{u}_{k+1}\|=\|\mathbf{u}_k\|=1$, the third equality comes from the Taylor's Theorem, and the fourth equality is a result of \eqref{equ:Jrelation}.
	
	Let $R(\theta_k\Delta\mathbf{u}_k)=(\mathbf{u}_{k+1}-\mathbf{u}_{k})^TJ_{\mathbf{v}_k}(\mathbf{u}_{k},\lambda_k)(\mathbf{u}_{k+1}-\mathbf{u}_{k})+E_k^T\mathbf{u}_{k+1}-\bar{E}_k$, the magnitude of $|R(\theta_k\Delta\mathbf{u}_k)|$ can be re-estimated as
	\begin{equation*}
		\begin{aligned}
			&|(\mathbf{u}_{k+1}-\mathbf{u}_{k})^TJ_{\mathbf{v}_k}(\mathbf{u}_{k},\lambda_k)(\mathbf{u}_{k+1}-\mathbf{u}_{k})+E_k^T\mathbf{u}_{k+1}-\bar{E}_k|\\
			\le & \|A_1+\beta_{12}\text{diag}(\mathbf{v}_k^{[2]})+3\beta_{11}\text{diag}(\mathbf{u}_k^{[2]})-\lambda_k I\|\|\mathbf{u}_{k+1}-\mathbf{u}_k\|^2+3\beta_{11}\|\mathbf{u}_{k+1}-\mathbf{u}_k\|^2\|\mathbf{u}_{k+1}\|\\
			&+3\beta_{11}\|\mathbf{u}_{k+1}-\mathbf{u}_k\|^2\\
			\le& (M_1+6\beta_{11})\|\mathbf{u}_{k+1}-\mathbf{u}_k\|^2,
		\end{aligned}
	\end{equation*}
	where $M_1$ is a positive constant determined by $A_1$, $\beta_{11}$ and $\beta_{12}$, since $\lambda_k$ is bounded as given in step 4 of Algorithm \ref{alg:ANNI}. Combining with
	\begin{equation*}
		\|\mathbf{u}_{k+1}-\mathbf{u}_k\|^2=2\left(1-\frac{1}{\|\mathbf{w}_{k+1}\|}\right)=2\left(1-\frac{1}{\sqrt{1+\|\theta_k\Delta\mathbf{u}_k\|^2}}\right)\le \|\theta_k\Delta\mathbf{u}_k\|^2,
	\end{equation*}
	the desired results follow with $M=M_1+6\beta_{11}$. $M$ is independent of $\mathbf{u}_k$ and determined by $A_1$, $\beta_{11}$ and $\beta_{12}$.
\end{proof}

\begin{theorem}\label{thm:theta}
	Given $\bar{\eta}>0$ and a unit vector $\mathbf{u}_k$, suppose $\mathbf{u}_k$ is not an eigenvector of $\mathcal{A}_{\mathbf{v}_k}(\mathbf{u}_k)$ and $\theta_k$ is generated by Algorithm \ref{alg:ANNI}. Then
	\begin{enumerate}[(i)]
		\item  $1\ge\theta_k\ge\min\{1,\frac{\bar{\eta}_k}{2}\}$, where
		\[\bar{\eta}_k=\frac{2\bar{\eta}}{(1+\bar{\eta})M\|\mathbf{w}_{k+1}\|\|\Delta\mathbf{u}_k\|^2}(\Delta\mathbf{u}_k^TJ_{\mathbf{v}_k}(\mathbf{u}_{k},\lambda_k)\Delta\mathbf{u}_k),\]
		and $M$ is a positive constant determined by $A_1$, $\beta_{11}$ and $\beta_{12}$. 
		\item The sequence $\{\theta_k\}$ is bounded below by a positive constant, that is $\theta_k\ge\xi>0$, where $\xi$ is determined by $A_1$, $\beta_{11}$ and $\beta_{12}$.
	\end{enumerate}
\end{theorem}
\begin{proof}
	(i) According to \eqref{equ:relation}, $\mathbf{u}_k^T\Delta\mathbf{u}_k=0$, hence $\|\mathbf{w}_{k+1}\|=\|\mathbf{u}_k+\theta_k\Delta\mathbf{u}_k\|=\sqrt{1+\|\theta_k\Delta\mathbf{u}_k\|^2}$. Combining with
	\[\mathbf{r}_{\mathbf{v}_k}(\mathbf{u}_{k},\lambda_k)^T\mathbf{u}_{k}=\mathbf{u}_{k}^T(\mathcal{A}_{\mathbf{v}_k}(\mathbf{u}_k)-\lambda_kI)\mathbf{u}_k> 0,\]
	we obtain that the first term in \eqref{equ:measure_d} is nonpositive. Therefore, 
	\[d_k(\theta_k)\le \frac{2\theta_k}{\|\mathbf{w}_{k+1}\|}\mathbf{r}_{\mathbf{v}_k}(\mathbf{u}_{k},\lambda_k)^T\Delta\mathbf{u}_k
	+R(\theta_k\Delta\mathbf{u}_k).\]
	As a result of \eqref{equ:relation}, we have
	\begin{equation*}
		\mathbf{r}_{\mathbf{v}_k}(\mathbf{u}_{k},\lambda_k)^T\Delta\mathbf{u}_k=(\delta_k\mathbf{u}_k-J_{\mathbf{v}_k}(\mathbf{u}_{k},\lambda_k)\Delta\mathbf{u}_k)^T\Delta\mathbf{u}_k=-\Delta\mathbf{u}_k^TJ_{\mathbf{v}_k}(\mathbf{u}_{k},\lambda_k)\Delta\mathbf{u}_k\le 0.
	\end{equation*}
	When $\mathbf{u_k}$ is not an eigenvector of $\mathcal{A}_{\mathbf{v}_k}(\mathbf{u_k})$, the above inequality strictly holds from Lemma \ref{lem:urelation} and the positive definiteness of $J_{\mathbf{v}_k}(\mathbf{u}_{k},\lambda_k)$ stated in Lemma \ref{lem:deltabound}. For any given $\bar{\eta}>0$, we have
	\begin{equation*}
		\begin{aligned}
			&d_k(\theta_k)\\
			\le& \frac{2\theta_k}{(1+\bar{\eta})\|\mathbf{w}_{k+1}\|}(-\Delta\mathbf{u}_k^TJ_{\mathbf{v}_k}(\mathbf{u}_{k},\lambda_k)\Delta\mathbf{u}_k)+\frac{2\bar{\eta}\theta_k}{(1+\bar{\eta})\|\mathbf{w}_{k+1}\|}(-\Delta\mathbf{u}_k^TJ_{\mathbf{v}_k}(\mathbf{u}_{k},\lambda_k)\Delta\mathbf{u}_k)\\
			&+M\|\theta_k\Delta\mathbf{u}_k\|^2,
		\end{aligned}
	\end{equation*}
	where $M$ is defined as in Lemma \ref{lem:measure_d}. Thus, if
	\[0<\theta_k\le \frac{2\bar{\eta}}{(1+\bar{\eta})M\|\mathbf{w}_{k+1}\|\|\Delta\mathbf{u}_k\|^2}(\Delta\mathbf{u}_k^TJ_{\mathbf{v}_k}(\mathbf{u}_{k},\lambda_k)\Delta\mathbf{u}_k),\]
	we have
	\begin{equation*}
		\frac{2\bar{\eta}\theta_k}{(1+\bar{\eta})\|\mathbf{w}_{k+1}\|}(-\Delta\mathbf{u}_k^TJ_{\mathbf{v}_k}(\mathbf{u}_{k},\lambda_k)\Delta\mathbf{u}_k)+M\|\theta_k\Delta\mathbf{u}_k\|^2\le 0.
	\end{equation*}
	and 
	\begin{equation}\label{equ:d_descent}
		d_k(\theta_k)\le \frac{2\theta_k}{(1+\bar{\eta})\|\mathbf{w}_{k+1}\|}(-\Delta\mathbf{u}_k^TJ_{\mathbf{v}_k}(\mathbf{u}_{k},\lambda_k)\Delta\mathbf{u}_k)<0.
	\end{equation}
	Since $\theta_k$ generated by Algorithm \ref{alg:ANNI} is produced by the halving procedure start with $\theta_k=1$ until $d_k(\theta_k)<0$, hence $\theta_k=1$ or $\theta_k\ge\frac{\bar{\eta}_k}{2}$, with 
	\[\bar{\eta}_k=\frac{2\bar{\eta}}{(1+\bar{\eta})M\|\mathbf{w}_{k+1}\|\|\Delta\mathbf{u}_k\|^2}(\Delta\mathbf{u}_k^TJ_{\mathbf{v}_k}(\mathbf{u}_{k},\lambda_k)\Delta\mathbf{u}_k).\]
	
	(ii) According to Lemma \ref{lem:deltabound}, $\lambda_{\min}(J_{\mathbf{v}_k}(\mathbf{u}_{k},\lambda_k))\ge \gamma>0$ and there exists $M_1>0$ such that $\|\mathbf{w}_{k+1}\|=\sqrt{1+\theta_k^2\|\Delta\mathbf{u}_k\|^2}\le M_1$ is bounded, thus we have
	\[\bar{\eta}_k\ge \frac{2\bar{\eta}}{(1+\bar{\eta})MM_1\|\Delta\mathbf{u}_k\|^2}(\gamma\|\Delta\mathbf{u}_k\|^2)=\frac{2\bar{\eta}\gamma}{(1+\bar{\eta})MM_1}>0.\]
	Let
	\begin{equation}\label{equ:xi}
		\xi=\min\left\{1,\frac{\bar{\eta}\gamma}{(1+\bar{\eta})MM_1}\right\},
	\end{equation}
	then $\theta_k$ is bounded below by a positive constant.
\end{proof}

\subsection{Convergence of alternating Newton-Noda iteration}
Let $\theta_k^{\mathbf{u}}$ and $\theta_k^{\mathbf{v}}$ denote the step sizes $\theta$ generated in Algorithm \ref{alg:ANNI} for $\mathbf{u}_k$ and $\mathbf{v}_k$, respectively. $\bar{\eta}_k^{\mathbf{u}}$ and $\bar{\eta}_k^{\mathbf{v}}$ are defined accordingly as $\bar{\eta}_k$ in Theorem \ref{thm:theta}.
\begin{theorem}\label{thm:descent}
	Given $\bar{\eta}>0$ and unit vectors $\mathbf{u}_k$ and $\mathbf{v}_k$, suppose $\mathbf{u}_k$ is not an eigenvector of $\mathcal{A}_{\mathbf{v}_k}(\mathbf{u}_k)$ and $\mathbf{v}_k$ is not an eigenvector of $\mathcal{A}_{\mathbf{u}_{k+1}}(\mathbf{v}_k)$, $\mathbf{u}_{k+1}$ and $\mathbf{v}_{k+1}$ are generated by Algorithm \ref{alg:ANNI}. If $\theta_k^{\mathbf{u}}\le\bar{\eta}_k^{\mathbf{u}}$ and $\theta_k^{\mathbf{v}}\le\bar{\eta}_k^{\mathbf{v}}$, then
	\begin{equation}\label{equ:whole_descent}
		f(\mathbf{u}_{k},\mathbf{v}_{k})-f(\mathbf{u}_{k+1},\mathbf{v}_{k+1})\ge C(\|\Delta\mathbf{u}_k\|^2+\|\Delta\mathbf{v}_k\|^2),
	\end{equation}
	where $C$ is a positive constant determined by $A_1,~A_2,~\beta_{11},~\beta_{12}$ and $\beta_{22}$.
\end{theorem}
\begin{proof}
	According to \eqref{equ:d_descent} and the parallel inequality with respect to $\mathbf{v}$, we have 
	\begin{equation*}
		\begin{aligned}
			f(\mathbf{u}_{k},\mathbf{v}_{k})-f(\mathbf{u}_{k+1},\mathbf{v}_{k+1})&=f(\mathbf{u}_{k},\mathbf{v}_{k})-f(\mathbf{u}_{k+1},\mathbf{v}_{k})+f(\mathbf{u}_{k+1},\mathbf{v}_{k})-f(\mathbf{u}_{k+1},\mathbf{v}_{k+1})\\
			&\ge\frac{2\theta_k^{\mathbf{u}}}{(1+\bar{\eta})\|\mathbf{w}_{k+1}\|}(\Delta\mathbf{u}_k^TJ_{\mathbf{v}_k}(\mathbf{u}_{k},\lambda_k)\Delta\mathbf{u}_k)\\
			&+\frac{2\theta_k^{\mathbf{v}}}{(1+\bar{\eta})\|\mathbf{z}_{k+1}\|}(\Delta\mathbf{v}_k^TJ_{\mathbf{u}_{k+1}}(\mathbf{v}_{k},\mu_k)\Delta\mathbf{v}_k)\\
			&\ge C(\|\Delta\mathbf{u}_k\|^2+\|\Delta\mathbf{v}_k\|^2),
		\end{aligned}
	\end{equation*}
	where $C$ is a positive constant determined by $A_1,~A_2,~\beta_{11},~\beta_{12}$ and $\beta_{22}$. The last inequality is a result of Lemma \ref{lem:deltabound} and the lower boundedness of $\theta_k^{\mathbf{u}}$ and $\theta_k^{\mathbf{v}}$ by Theorem \ref{thm:theta}.
\end{proof}

\begin{remark}
	We deduced \eqref{equ:whole_descent} under the condition $\theta_k^{\mathbf{u}}\le\bar{\eta}_k^{\mathbf{u}}$ and $\theta_k^{\mathbf{v}}\le\bar{\eta}_k^{\mathbf{v}}$ which can be estimated roughly by \eqref{equ:xi}. However, for practical purpose, we simply update $\theta_k$ by the halving procedure as presented in Algorithm \ref{alg:ANNI}, that is $\theta_k\leftarrow\theta_k/2$, until $d_k(\theta_k)<0$. According to Theorem \ref{thm:theta}, we can always find such a $\theta_k$ within finite halving steps.
\end{remark}

Now, we are ready to analyze the global convergence of ANNI in the following theorem.
\begin{theorem}\label{thm:global_convergence}
	Given any unit positive initial points $\mathbf{u}_0$ and $\mathbf{v}_0$, assume $\{(\mathbf{u}_k,\mathbf{v}_k)\}$ is generated by Algorithm \ref{alg:ANNI}, we have the following results:
	\begin{enumerate}[(i)]
		\item Any limit point of the sequence $\{(\mathbf{u}_k,\mathbf{v}_k)\}$ is a critical point of \eqref{equ:problem}.
		\item If $A_1$ and $A_2$ are irreducible nonsingular $M$-matrices and $\tau_2$ satisfies \eqref{equ:tau_for_positive}, then the limit point is a positive critical point of \eqref{equ:problem}. Furthermore, if the nonnegative eigenpair of CNEPv\eqref{equ:CNEPv} is unique, the whole sequence $\{(\mathbf{u}_k,\mathbf{v}_k)\}$ converges to the positive global optimum of \eqref{equ:problem}.
	\end{enumerate}
\end{theorem}
\begin{proof}
	(i) Since $f(\mathbf{u}_k,\mathbf{v}_k)$ is bounded below with the spherical constraints, we have $\{f(\mathbf{u}_k,\mathbf{v}_k)\}$ is convergent and 
	\[\lim_{k\rightarrow\infty}\Delta\mathbf{u}_k=0,~\lim_{k\rightarrow\infty}\Delta\mathbf{v}_k=0,\]
	as a result of Theorem \ref{thm:descent}. Let $\delta_k^{\mathbf{u}}$ and $\delta_k^{\mathbf{v}}$ denote the $\delta_k$ about $\mathbf{u}_k$ and $\mathbf{v}_k$, respectively.
	From Lemma \ref{lem:urelation}, we have
	\[\lim_{k\rightarrow\infty}\left(\mathbf{r}_{\mathbf{v}_k}(\mathbf{u}_k,\lambda_k)-\delta_k^{\mathbf{u}}\mathbf{u}_k\right)=0,~\lim_{k\rightarrow\infty}\left(\mathbf{r}_{\mathbf{u}_{k+1}}(\mathbf{v}_k,\mu_k)-\delta_k^{\mathbf{v}}\mathbf{v}_k\right)=0.\]
	On the other hand, $\|\mathbf{u}_{k+1}-\mathbf{u}_k\|\le \|\theta_k\Delta\mathbf{u}_k\|$, thus $\lim_{k\rightarrow\infty}\|\mathbf{u}_{k+1}-\mathbf{u}_k\|=0$,
	and 
	\[\lim_{k\rightarrow\infty}\left(\mathbf{r}_{\mathbf{u}_{k}}(\mathbf{v}_k,\mu_k)-\delta_k^{\mathbf{v}}\mathbf{v}_k\right)=0.\]
	Since $\{(\mathbf{u}_k,\mathbf{v}_k,\lambda_k,\mu_k,\delta_k^{\mathbf{u}},\delta_k^{\mathbf{v}})\}$ is bounded, there exists a subsequence converging to $(\mathbf{u}_*,\mathbf{v}_*,\lambda_*,\mu_*,\delta_*^{\mathbf{u}},\delta_*^{\mathbf{v}})$ with $\|\mathbf{u}_*\|=\|\mathbf{v}_*\|=1$. This leads to
	\[\mathcal{A}_{\mathbf{v}_*}(\mathbf{u}_*)\mathbf{u}_*=(\lambda_*+\delta_*^{\mathbf{u}})\mathbf{u}_*,~\mathcal{A}_{\mathbf{u}_*}(\mathbf{v}_*)\mathbf{v}_*=(\mu_*+\delta_*^{\mathbf{v}})\mathbf{v}_*.\]
	Hence, $(\mathbf{u}_*,\mathbf{v}_*)$ is a critical point of \eqref{equ:problem}.
	
	(ii) If $A_1$ and $A_2$ are irreducible nonsingular $M$-matrices and $\tau_2$ satisfies \eqref{equ:tau_for_positive}, then $\mathbf{u}_k$ and $\mathbf{v}_k$ are positive according to Theorem \ref{thm:positivity}. Thus the limit point $(\mathbf{u}_*,\mathbf{v}_*)$ of $\{(\mathbf{u}_k,\mathbf{v}_k)\}$ is a nonnegative critical point of $\eqref{equ:problem}$. The rest of the proof is similar to that in Theorem \ref{thm:alm}.
	
\end{proof}

\section{Extensions}\label{sec:general}
\subsection{The strategy for choosing $\lambda_k$ and $\mu_k$}
Through the above arguments, the convergence of Algorithm \ref{alg:ANNI} for \eqref{equ:problem} is a direct consequence of the sufficient descent of $f(\mathbf{u},\mathbf{v})$ in each iteration as \eqref{equ:whole_descent}. To reach this purpose, the key is to guarantee that both $\lambda_k$ and $\mu_k$ are bounded and  $\lambda_{\min}(J_{\mathbf{v}_k}(\mathbf{u}_k,\lambda_k))$, $\lambda_{\min}(J_{\mathbf{u}_{k+1}}(\mathbf{v}_k,\mu_k))\ge\gamma$ for some positive constant, according to the proof procedure. In Algorithm \ref{alg:ANNI}, these conditions are satisfied by choosing $\lambda_k,~\mu_k\in[\tau_1,\tau_2]$, where $$\tau_1<\tau_2<\min\{\lambda_{\min}( \mathcal{A}_{\mathbf{v}}(\mathbf{u})),\lambda_{\min}( \mathcal{A}_{\mathbf{u}}(\mathbf{v}))\}$$ for all unit $\mathbf{u},\mathbf{v}$.

In particular, if the two-component BEC problem \eqref{equ:problem} is a result of the  finite difference scheme discretizing \eqref{equ:spin-1/2}, see section \ref{sec:numerical} for details, then $A_1$ and $A_2$ are irreducible nonsingular $M$-matrices. We can select the $\lambda_k$ and $\mu_k$ as 
\begin{equation}\label{equ:select}
	\begin{aligned}
		\lambda_k&=\max\left\{\tau_1,\min\left(\frac{\mathcal{A}_{\mathbf{v}_k}(\mathbf{u}_k)\mathbf{u}_k}{\mathbf{u}_k}\right)\right\},\\
		\mu_k&=\max\left\{\tau_1,\min\left(\frac{\mathcal{A}_{\mathbf{u}_{k+1}}(\mathbf{v}_k)\mathbf{v}_k}{\mathbf{v}_k}\right)\right\},
	\end{aligned}
\end{equation}
where $\tau_1<\min\{\lambda_{\min}( \mathcal{A}_{\mathbf{v}}(\mathbf{u})),\lambda_{\min}( \mathcal{A}_{\mathbf{u}}(\mathbf{v}))\}$ for all positive unit $\mathbf{u},\mathbf{v}$. We can estimate
$$\lambda_{\min}( \mathcal{A}_{\mathbf{v}}(\mathbf{u}))=\lambda_{\min}\left(\beta_{11}\text{diag}(\mathbf{u}^{[2]})+(A_1+\beta_{12}\text{diag}(\mathbf{v}^{[2]}))\right)\ge \lambda_{\min}(A_1),$$  and let $\tau_1=\lambda_{\min}(A_1)-\epsilon$, where $\epsilon$ is a given positive constant. Since $\mathcal{A}_{\mathbf{v}}(\mathbf{u})$ and $\mathcal{A}_{\mathbf{u}}(\mathbf{v})$ are always positive definite in the BEC problems considered here, $\tau_1=0$ can be an option.

\begin{lemma}\label{lem:M of J}
	Let $\tau_1\le \tau_2$ satisfies \eqref{equ:tau_for_positive}, suppose $\{(\mathbf{u}_k,\mathbf{v}_k)\}$ is generated by ANNI with $\lambda_k$ and $\mu_k$ selected as \eqref{equ:select}, solving \eqref{equ:problem} with the finite difference scheme. Then $J_{\mathbf{v}_k}(\mathbf{u}_k,\lambda_k)$ is a positive definite irreducible $M$-matrix for any $k$.
\end{lemma}
\begin{proof}
	Since \eqref{equ:problem} is presented with the finite difference scheme, $A_1$ is a irreducible $M$-matrix. Then for a given unit vector $\mathbf{u}_k>0$,  $\mathcal{A}_{\mathbf{v}_k}(\mathbf{u}_k)=\beta_{11}\text{diag}(\mathbf{u}_k^{[2]})+(A_1+\beta_{12}\text{diag}(\mathbf{v}_k^{[2]}))$ is also a nonsingular irreducible $M$-matrix.
	
	If $\lambda_k=\tau_1$, we have
	$$J_{\mathbf{v}_k}(\mathbf{u}_k,\lambda_k)=\mathcal{A}_{\mathbf{v}_k}(\mathbf{u}_k)-\tau_1 I+2\beta_{11}\text{diag}(\mathbf{u}^{[2]}).$$ Then $J_{\mathbf{v}_k}(\mathbf{u}_k,\lambda_k)$ is obviously a positive definite $M$-matrix since $$\tau_1<\min\{\lambda_{\min}( \mathcal{A}_{\mathbf{v}}(\mathbf{u})),\lambda_{\min}( \mathcal{A}_{\mathbf{u}}(\mathbf{v}))\}$$ for all positive unit $\mathbf{u},\mathbf{v}$. Hence, $\mathbf{u}_{k+1}>0$ is a result of Theorem \ref{thm:positivity}.
	
	If $\lambda_k=\min\left(\frac{\mathcal{A}_{\mathbf{v}_k}(\mathbf{u}_k)\mathbf{u}_k}{\mathbf{u}_k}\right)$, for given $\mathbf{u}_k>0$, we have
	\[J_{\mathbf{v}_k}(\mathbf{u}_k,\lambda_k)\mathbf{u}_k=\mathcal{A}_{\mathbf{v}_k}(\mathbf{u}_k)\mathbf{u}_k-\lambda_k\mathbf{u}_k+2\beta_{11}\mathbf{u}_k^{[3]}\ge 2\beta_{11}\mathbf{u}_k^{[3]}>0.\]
	Then $J_{\mathbf{v}_k}(\mathbf{u}_k,\lambda_k)$ is also a positive definite irreducible $M$-matrix from Theorem \ref{thm:M-matrix}. On the other hand, 
	$$\mathbf{r}_{\mathbf{v}_k}(\mathbf{u}_k,\lambda_k)=\mathcal{A}_{\mathbf{v}_k}(\mathbf{u}_k)\mathbf{u}_k-\lambda_k\mathbf{u}_k\ge0.$$ We have $\mathbf{u}_k^TJ_{\mathbf{v}_k}^{-1}(\mathbf{u}_k,\lambda_k)\mathbf{r}_{\mathbf{v}_k}(\mathbf{u}_k,\lambda_k)\ge0$ by Theorem 2.3, then $\delta_k\ge 0$ from \eqref{equ:delta2}. Hence we can also obtain $\mathbf{u}_{k+1}>0$ as the proof for Theorem \ref{thm:positivity}. Therefore, $J_{\mathbf{v}_k}(\mathbf{u}_k,\lambda_k)$ is a positive definite irreducible $M$-matrix for any $k$ given $\mathbf{u}_0>0$.
\end{proof}

\begin{corollary}\label{cor:finite}
	Let $\tau_1\le \tau_2$ satisfying \eqref{equ:tau_for_positive}, suppose $\{(\mathbf{u}_k,\mathbf{v}_k)\}$ is generated by ANNI with $\lambda_k$ and $\mu_k$ selected as \eqref{equ:select}, solving \eqref{equ:problem} with the finite difference scheme. Then any limit point of this sequence is a positive critical point of $\eqref{equ:problem}.$
\end{corollary}
\begin{proof}
	From Lemma \ref{lem:M of J}, $J_{\mathbf{v}_k}(\mathbf{u}_k,\lambda_k)$ is always a positive definite and irreducible $M$-matrix and $\{\mathbf{u}_{k}\}$ is a positive sequence.
	We only need to prove that there exists a positive constant $\gamma$, such that for any $k$ $$\lambda_{\min}(J_{\mathbf{v}_k}(\mathbf{u}_k,\lambda_k))\ge\gamma.$$
	
	If $\lambda_k=\tau_1$, then $$\lambda_{\min}(J_{\mathbf{v}_k}(\mathbf{u}_k,\lambda_k))\ge\lambda_{\min}(\mathcal{A}_{\mathbf{v}_k}(\mathbf{u}_k))-\tau_1>\gamma_1>0$$ for some constant $\gamma_1$. 
	
	If $\lambda_k=\tau_1$ in only a finite number of iterations, we must have that $\min(\mathbf{u}_k)>\gamma_2$ for some positive constant $\gamma_2$. Otherwise, there exists a subsequence $\{\mathbf{u}_{k_j}\}$ converging to $\mathbf{u}_*$ with $(\mathbf{u_*})_i=0$ for any $i\in I$, where $I$ is a nonempty subset of $\{1,2,\cdots,n\}$ and $\{\mathbf{v}_{k_j}\}$ converging to $\mathbf{v}_*$. With $\lambda_{k_j}=\min\left(\frac{\mathcal{A}_{\mathbf{v}_{k_j}}(\mathbf{u}_{k_j})\mathbf{u}_{k_j}}{\mathbf{u}_{k_j}}\right)$, then we have for any $i\in I$,
	\[0\ge(A_1\mathbf{u}_*)_i=\lim_{j\rightarrow\infty} \left(\mathcal{A}_{\mathbf{v}_{k_j}}(\mathbf{u}_{k_j})\mathbf{u}_{k_j}\right)_i\ge\lim_{j\rightarrow\infty}\lambda_{k_j}(\mathbf{u}_{k_j})_i\ge\lim_{j\rightarrow\infty}\tau_1(\mathbf{u}_{k_j})_i =0,\]
	which contradicts the irreducibility of $A_1$. Therefore, there exists ${\gamma_2}$, such that $$\lambda_{\min}(J_{\mathbf{v}_k}(\mathbf{u}_k,\lambda_k))\ge\lambda_{\min}(2\beta_{11}\text{diag}(\mathbf{u}_k^{[2]}))\ge 2\beta_{11}{\gamma_2}^2$$ whenever $\lambda_k=\min\left(\frac{\mathcal{A}_{\mathbf{v}_k}(\mathbf{u}_k)\mathbf{u}_k}{\mathbf{u}_k}\right)$.
	
	Finally, we can conclude that there exists a positive constant $\gamma$ such that $$\lambda_{\min}(J_{\mathbf{v}_k}(\mathbf{u}_k,\lambda_k))\ge\gamma.$$ The rest of the proof is the same as that in section \ref{sec:ANNI}.
\end{proof}

\begin{remark}
	According to Corollary \ref{cor:finite}, the limit point $(\mathbf{u}_*,\mathbf{v}_*)$ of $\{(\mathbf{u}_k,\mathbf{v}_k)\}$ is a positive critical point of \eqref{equ:problem}, which satisfies the CNEPv\eqref{equ:CNEPv}. Then $\lambda_*=\lambda_{\min}(\mathcal{A}_{\mathbf{v}_*}(\mathbf{u}_*))=\min\left(\frac{\mathcal{A}_{\mathbf{v}_*}(\mathbf{u}_*)\mathbf{u}_*}{\mathbf{u}_*}\right)>\tau_1$. Therefore, if $(\mathbf{u}_k,\mathbf{v}_k)$ is sufficiently close to $(\mathbf{u}_*,\mathbf{v}_*)$, $\lambda_k$ can be computed by $\min\left(\frac{\mathcal{A}_{\mathbf{v}_k}(\mathbf{u}_k)\mathbf{u}_k}{\mathbf{u}_k}\right)$, which is the same as the $\lambda_k$ used in NNI. From this point of view, ANNI with \eqref{equ:select} may be promising to achieve a locally quadratically convergence rate as \cite{liu2020positivity} in some cases, which needs further study.
\end{remark}

\subsection{Multi-block problems}

We have discussed the algorithms and their convergence properties for the ``two-block" quartic-quadratic nonconvex minimization problem with spherical constraints in the form of \eqref{equ:two-block}:
\begin{equation}\label{equ:two-block}
	\begin{aligned}
		\underset{\mathbf{u}^1,\mathbf{u}^2\in\mathbb{R}^n}{\min}~&f(\mathbf{u}^1,\mathbf{u}^2)=\sum_{j=1}^2h(\mathbf{u}^j)+\beta_{12}\sum_{i=1}^n(\mathbf{u}^1)_i^2(\mathbf{u}^2)_i^2\\	
		\text{s.t.}~&{\mathbf{u}^1}^T\mathbf{u}^1=1,~{\mathbf{u}^2}^T\mathbf{u}^2=1,
	\end{aligned}
\end{equation}
where $$h(\mathbf{u})=\frac{\beta}{2}\sum_{i=1}^nu_i^4+\mathbf{u}^TA\mathbf{u}.$$ The ``two-block" here refers to the structure of $\sum_{j=1}^2h(\mathbf{u}^j)$. In this section, we will consider the  ``multi-block" spherical constraints nonconvex problem with more general $h(\mathbf{u})$ presented as
\begin{equation}\label{equ:multi-block}
	\begin{aligned}
		\underset{\mathbf{u}^j\in\mathbb{R}^n}{\min}~&f(\mathbf{u}^1,\mathbf{u}^2,\cdots,\mathbf{u}^m)=\sum_{j=1}^mh(\mathbf{u}^j)+\sum_{j,s=1,j\neq s}^m\frac{\beta_{js}}{2}\sum_{i=1}^n(\mathbf{u}^j)_i^2(\mathbf{u}^s)_i^2\\	
		\text{s.t.}~&{\mathbf{u}^j}^T\mathbf{u}^j=1,~j=1,2,\cdots,m,
	\end{aligned}	
\end{equation}
where
\begin{equation*}
	h(\mathbf{u})=2\tilde{h}(\mathbf{u})+\mathbf{u}^TA\mathbf{u}.
\end{equation*}

It is straightforward to generalize ANNI to \eqref{equ:multi-block} under mild conditions. At the $k$th iteration, we apply the one-step modified NNI to $\mathbf{u}^j,~j=1,2,\cdots,m$ alternatively as Algorithm \ref{alg:ANNI}, which consists of four main steps:
\begin{equation*}\label{equ:ANNI}
	\begin{aligned}
		\lambda_k^j&\in[\tau_1,\tau_2],\\
		F'_{\mathbf{v}_k^j}(\mathbf{u}_k^j,\lambda_k^j)\begin{bmatrix}\Delta\mathbf{u}_k^j\\\delta_k^j\end{bmatrix}&=-F_{\mathbf{v}_k^j}(\mathbf{u}_k^j,\lambda_k^j),\\
		d_k^j(\theta_k^j)&={f}_{\mathbf{v}^j_k}(\mathbf{u}_{k+1}^j)-{f}_{\mathbf{v}^j_k}(\mathbf{u}_{k}^j)<0,\\
		\mathbf{u}_{k+1}^j&=\frac{\mathbf{u}_k^j+\theta_k^j\Delta\mathbf{u}_k^j}{\|\mathbf{u}_k^j+\theta_k^j\Delta\mathbf{u}_k^j\|},
	\end{aligned}
\end{equation*}
where $\mathbf{v}_k^j=(\mathbf{u}_{k+1}^1,\cdots,\mathbf{u}_{k+1}^{j-1},\mathbf{u}_{k}^{j+1},\cdots,\mathbf{u}_k^m)$. Let 
\[\mathcal{A}_{\mathbf{v}_k^j}(\mathbf{u})=\text{diag}\left(\frac{\nabla\tilde{h}(\mathbf{u})}{\mathbf{u}}\right)+A+\sum_{s\neq j}\frac{\beta_{sj}+\beta_{js}}{2}\text{diag}\left((\mathbf{v}_k^j)_s^{[2]}\right).\]
Then $\tau_1,~\tau_2$, $F'_{\mathbf{v}_k^j}(\mathbf{u},\lambda)$, $\mathbf{r}_{\mathbf{v}_k^j}(\mathbf{u},\lambda),~J_{\mathbf{v}_k^j}(\mathbf{u},\lambda)$ and ${f}_{\mathbf{v}_k^j}(\mathbf{u})$ are defined correspondingly. Hence, it is natural to obtain the following theorem about the convergence result of ANNI for \eqref{equ:multi-block}.

\begin{theorem}\label{thm:multi-block}
	Suppose that $\tilde{h}$ is twice continuously differentiable and $\nabla^2\tilde{h}$ is locally Lipschitz continuous over the unit ball. Given any unit positive initial points $\mathbf{u}^j,~j=1,2,\cdots,m$, assume the sequence $\{(\mathbf{u}_k^1,\cdots,\mathbf{u}_k^m)\}$ is generated by ANNI for \eqref{equ:multi-block}, we have the following results:
	\begin{enumerate}[(i)]
		\item Any limit point of the sequence $\{(\mathbf{u}_k,\mathbf{v}_k)\}$ is a critical point of \eqref{equ:multi-block}.
		\item If $A$ is an irreducible nonsingular $M$-matrix, $J_{\mathbf{v}}(\mathbf{u},\lambda)$ is a nonsingular $M$-matrix, and $\nabla^2\tilde{h}(\mathbf{u})\mathbf{u}-\nabla\tilde{h}(\mathbf{u})>0$ for any $\mathbf{u}>0$, then the limit point is a nonnegative critical point of \eqref{equ:multi-block}. 
%		$\lambda_k^j$ can be computed as
%		\[\lambda_k^j=\max\left(\tau_1,\min\left(\frac{\mathcal{A}_{\mathbf{v}^j_k}(\mathbf{u}_k^j)\mathbf{u}_k^j}{\mathbf{u}_k^j}\right)\right).\]
	\end{enumerate}
\end{theorem}
\begin{proof}
	From the proof process of Theorem \ref{thm:global_convergence}, it is critical to obtain the sufficient descent of $d_k^j(\theta_k^j)$. And the key step for this is to obtain the formulation as \eqref{equ:measure_d}. For any $\mathbf{u}_k^j,~\|\mathbf{u}_k^j\|=1$, let $\mathbf{w}_{k+1}^j=\mathbf{u}_k^j+\theta_k^j\Delta\mathbf{u}_k^j$, we have
	\begin{equation*}
		\begin{aligned}
			&d_k^j(\theta_k^j)={f}_{\mathbf{v}^j_k}(\mathbf{u}_{k+1}^j)-{f}_{\mathbf{v}^j_k}(\mathbf{u}_{k}^j)\\
			=&\left[(\mathcal{A}_{\mathbf{v}_{k}^j}(\mathbf{u}_{k+1}^j)\mathbf{u}_{k+1}^j)^T\mathbf{u}_{k+1}^j-(\mathcal{A}_{\mathbf{v}_{k}^j}(\mathbf{u}_{k}^j)\mathbf{u}_{k}^j)^T\mathbf{u}_{k}^j\right]\\
			&+\left[2\tilde{h}(\mathbf{u}_{k+1}^j)-\nabla\tilde{h}(\mathbf{u}_{k+1}^j)^T\mathbf{u}_{k+1}^j-\left(2\tilde{h}(\mathbf{u}_{k}^j)-\nabla\tilde{h}(\mathbf{u}_{k}^j)^T\mathbf{u}_{k}^j\right)\right]\\
			=&\left[\mathbf{r}_{\mathbf{v}_k^j}(\mathbf{u}_{k+1}^j,\lambda_k^j)^T\mathbf{u}_{k+1}^j-\mathbf{r}_{\mathbf{v}_k^j}(\mathbf{u}_{k}^j,\lambda_k^j)^T\mathbf{u}_{k}^j\right]
			+\left[\left(\nabla\tilde{h}(\mathbf{u}_k^j)-\nabla^2\tilde{h}(\mathbf{u}_k^j)\mathbf{u}_k^j\right)^T(\mathbf{u}_{k+1}^j-\mathbf{u}_k^j)+\bar{E}_k^j\right]\\
			=&[\mathbf{r}_{\mathbf{v}_k^j}(\mathbf{u}_{k}^j,\lambda_k^j)+J_{\mathbf{v}_k^j}(\mathbf{u}_{k}^j,\lambda_k^j)(\mathbf{u}_{k+1}^j-\mathbf{u}_k^j)+E_k^j]^T\mathbf{u}_{k+1}^j-\mathbf{r}_{\mathbf{v}_k^j}(\mathbf{u}_{k}^j,\lambda_k^j)^T\mathbf{u}_{k}^j\\
			&-\left[(J_{\mathbf{v}_k^j}(\mathbf{u}_{k}^j,\lambda_k^j)\mathbf{u}_k^j-\mathbf{r}_{\mathbf{v}_k^j}(\mathbf{u}_{k}^j,\lambda_k))^T(\mathbf{u}_{k+1}^j-\mathbf{u}_k^j)-\bar{E}_k^j\right]\\
			%			=&\mathbf{r}_{\mathbf{v}_k}(\mathbf{u}_{k},\lambda_k)^T\left(\frac{\mathbf{u}_k+\theta_k\Delta\mathbf{u}_k}{\|\mathbf{w}_{k+1}\|}\right)-\mathbf{r}_{\mathbf{v}_k}(\mathbf{u}_{k},\lambda_k)^T\mathbf{u}_{k}+\mathbf{u}_{k+1}^TJ_{\mathbf{v}_k}(\mathbf{u}_{k},\lambda_k)(\mathbf{u}_{k+1}-\mathbf{u}_k)+E_k^T\mathbf{u}_{k+1}\\
			%			&-\left[(J_{\mathbf{v}_k}(\mathbf{u}_{k},\lambda_k)\mathbf{u}_k-\mathbf{r}_{\mathbf{v}_k}(\mathbf{u}_{k},\lambda_k))^T(\mathbf{u}_{k+1}-\mathbf{u}_k)+\bar{E}_k\right]\\
			=&2\left(\frac{1}{\|\mathbf{w}_{k+1}^j\|}-1\right)\mathbf{r}_{\mathbf{v}_k^j}(\mathbf{u}_{k}^j,\lambda_k^j)^T\mathbf{u}_{k}^j+\frac{2\theta_k^j}{\|\mathbf{w}_{k+1}^j\|}\mathbf{r}_{\mathbf{v}_k^j}(\mathbf{u}_{k}^j,\lambda_k^j)^T\Delta\mathbf{u}_k^j\\
			&+(\mathbf{u}_{k+1}^j-\mathbf{u}_{k}^j)^TJ_{\mathbf{v}_k^j}(\mathbf{u}_{k}^j,\lambda_k^j)(\mathbf{u}_{k+1}^j-\mathbf{u}_{k}^j)+E_k^T\mathbf{u}_{k+1}^j+\bar{E}_k^j\\
		\end{aligned}
	\end{equation*}
	and $\mathbf{r}_{\mathbf{v}_k^j}(\mathbf{u}_{k}^j,\lambda_k^j)^T\Delta\mathbf{u}_k^j=-(\Delta\mathbf{u}_k^j)^TJ_{\mathbf{v}_k^j}(\mathbf{u}_{k}^j,\lambda_k^j)\Delta\mathbf{u}_k^j<0$. The rest of the proof is similar to that of ANNI for \eqref{equ:two-block}, so we omit it.
\end{proof}

\begin{remark}
	It is obvious that $\tilde{h}(\mathbf{u})=\sum_{i=1}^nu_i^p$ for any $p>2$ satisfies the conditions stated in Theorem \ref{thm:multi-block} (ii). According to \cite[Lemma 3.1]{du2022newton}, these conditions are also satisfied by $\tilde{h}(\mathbf{u})=(\mathbf{u}^{[2]})^TA\mathbf{u}^{[2]}$ which comes from the modified GPEs \cite{bao2019computing}, whose gradient $\nabla\tilde{h}(\mathbf{u})=(A\mathbf{u}^{[2]})\circ\mathbf{u}$ and ``$\circ$" is the Hadamard product, as well as $\tilde{h}(\mathbf{u})=\sum_{i=1}^n(u_i^2-\ln(1+\frac{u_i^2}{a_i}))$ with $a_i>0$ arising from the saturable nonlinear Schr\"odinger equation \cite{liu2020positivity}. 
\end{remark}

%\begin{remark}
%	Note that $J_{\mathbf{v}_k}(\mathbf{u}_k,\lambda_k)$ has nice structure with (block) tridiagonal form, if the multi-component BEC problem \eqref{equ:functional} is discretized by the finite difference scheme. According to \eqref{equ:delta} and \eqref{equ:Delta}, in this case, $\delta_k$ and $\Delta\mathbf{u}_k$ can be computed easily by solving efficiently the (block) tridiagonal linear systems $J_{\mathbf{v}_k}(\mathbf{u}_k,\lambda_k)\mathbf{d}_1=\mathbf{u}_k$ and $J_{\mathbf{v}_k}(\mathbf{u}_k,\lambda_k)\mathbf{d}_2=\mathbf{r}_{\mathbf{v}_k}(\mathbf{u}_k,\lambda_k)$.
%\end{remark}

\section{Numerical experiments}\label{sec:numerical}
In this section, we evaluate the numerical performance of ALM \eqref{equ:ALM} and ANNI. We first compare the ALM \eqref{equ:ALM} and ANNI with the popular gradient flow method implemented by the backward Euler finite difference BEFD \cite{bao2004ground},  through testing some spin-1/2 BEC problems discretized by the finite difference method. Then we apply ALM \eqref{equ:ALM} and ANNI to compute the ground states of special spin-1 and spin-2 BEC problems discretized by the Fourier pseudo-spectral scheme. Since it has been pointed out that the state-of-art Riemannian optimization algorithm ARNT \cite{tian2020ground} is more efficient than other methods proposed for computing ground states of spin-$F$ multi-component BECs in the literature, a comparison between ARNT with our proposed algorithms is also presented. All experiments were performed on a Lenovo laptop with Gen Intel(R) Core(TM) i7 at 2.3 GHz and 32 GB memory using Matlab R2022a.

\subsection{Implementation details}
We use NNI \cite{du2022newton} and NRI \cite{huang2022newton} to solve each subproblem in ALM \eqref{equ:ALM} for problems with the finite difference scheme and the pseudo-spectral scheme, respectively. Unless specifically mentioned, the stopping criterion used for the numerical tests can be described as $\|\text{grad}f(\mathbf{u}_k,\mathbf{v}_k)\|_2\le 10^{-6}$, where $\|\text{grad}f(\mathbf{u}_k,\mathbf{v}_k)\|_2$ is the norm of the Riemannian gradient, which can be calculated as:
\[2\sqrt{\|\mathcal{A}_{\mathbf{v}_k}(\mathbf{u}_k)\mathbf{u}_k-\left(\mathbf{u}_k^T\mathcal{A}_{\mathbf{v}_k}(\mathbf{u}_k)\mathbf{u}_k\right)\mathbf{u}_k\|^2+\|\mathcal{A}_{\mathbf{u}_k}(\mathbf{v}_k)\mathbf{v}_k-\left(\mathbf{v}_k^T\mathcal{A}_{\mathbf{u}_k}(\mathbf{v}_k)\mathbf{v}_k\right)\mathbf{v}_k\|^2}.\]
The maximum number of iterations for ANNI and ALM is 200. For convenience, the initial data is $\mathbf{u}_0=\frac{1}{\sqrt{n}}[1,\cdots,1]^T$, $\mathbf{v}_0=\frac{1}{\sqrt{n}}[1,\cdots,1]^T\in\mathbb{R}^n$.

\subsubsection{Preconditioner}
The main step in ANNI is to solve the linear system \eqref{equ:linear} involving $J_{\mathbf{v}}(\mathbf{u},\lambda)$. According to \eqref{equ:delta2} and \eqref{equ:Delta2}, we can compute \eqref{equ:linear} through solving 
\begin{equation}\label{equ:computing}
	J_{\mathbf{v}}(\mathbf{u},\lambda)\mathbf{y}_1=\mathbf{u},~J_{\mathbf{v}}(\mathbf{u},\lambda)\mathbf{y}_2=\mathbf{r}_{\mathbf{v}}(\mathbf{u},\lambda).
\end{equation}
For the BEC problem, note that $J_{\mathbf{v}}(\mathbf{u},\lambda)$ is the sum of a discretized negative Laplace operator matrix and diagonal matrices. Hence, we compute \eqref{equ:computing} as follows:
\begin{enumerate}[(1)]
	\item If we consider problems with the finite difference scheme, $J_{\mathbf{v}}(\mathbf{u},\lambda)$ has the nice block tridiagonal structure and there are mature methods such as block LU \cite{golub2013matrix} to compute it directly. 
	\item As for the discretized problem with the Fourier pseudo-spectral scheme, we use the conjugate gradient (CG) method \cite[chapter 7.2.4]{yang2010nonlinear} to compute it, and different  precondition matrices are used to improve the performance of algorithms. We simply list three types of preconditioner, which are usually used solving BEC problems \cite{antoine2017efficient} as follows.
\end{enumerate}

\begin{itemize}
	\item  \textbf{Kinetic energy preconditioner.} $P=(c-\Delta/2)^{-1}$, where $\Delta$ is a discretized matrix of the Laplace operator and $c$ is a positive shifting constant to construct an invertible operator. In our experiments, $c$ is selected to be 3 or 30.
	\item \textbf{Potential energy preconditioner.} We just give the form of such preconditioner for $\mathbf{u}$. $P_{V}=\left(c+V+\beta_{12}\text{diag}(\mathbf{v}^{[2]})+\beta_{11}\text{diag}(\mathbf{u}^{[2]})-\lambda I\right)^{-1}$, where $c$ is the same as the one in $P_{\Delta}$, $V$ is the discretized diagonal matrix of potential function and $\lambda$ is given as in Algorithm \ref{alg:ANNI}. $P_{V}$ is diagonal in real space.
	\item \textbf{Combined preconditioner.} $P_C=P_V^{1/2}PP_V^{1/2}$.
\end{itemize}

%Figure \ref{fig:precondition} shows the performance of ANNI utilizing different preconditioner matrices for Example \ref{exm:spin-1}. We have not figured out which preconditioner is the best for all examples, although it indeed has an important impact on the efficiency of ANNI. 
The preconditioner will affect not only the convergence rate but also the efficiency of solving linear systems. There is a trade-off between them. For most of our tests, we find that the combined preconditioner is the most stable and effective. While for the examples in 3D cases with a large computational domain, the potential energy preconditioner seems to perform better. We have not figured out the best preconditioner for all examples. For a detailed discussion of the properties and numerical performance of these three preconditioners, please refer to \cite{antoine2017efficient}.

%\begin{figure}[htb]
%	\centering
%	\subfigure[$U={[}-16,16{]}^2$, $\beta_0=0.3$, $\beta_1=0.1$, $M=0.9$]{
%		\begin{minipage}[t]{0.6\linewidth}
%			\centering
%			\includegraphics[width=3in]{precond_c30_L16_beta0.3_M0.9.eps}
%			%\caption{fig1}
%		\end{minipage}%
%	}%
%	\subfigure[$U={[}-16,16{]}^2$, $\beta_0=300$, $\beta_1=100$, $M=0.5$]{
%		\begin{minipage}[t]{0.5\linewidth}
%			\centering
%			\includegraphics[width=3in]{precond_L16_beta300_M0.5.eps}
%			%\caption{fig2}
%		\end{minipage}%
%	}%
%	\centering
%	\caption{The performance of ANNI with different preconditioner matrices.}
%	\label{fig:precondition}
%\end{figure}

\subsubsection{Choice of $\lambda_k$ and $\mu_k$}
In addition to the linear system solver, $\lambda_k$ and $\mu_k$ are the major factors that influence the convergence rate of ANNI by affecting the condition of $J_{\mathbf{v}_k}(\mathbf{u}_k,\lambda_k)$ and $J_{\mathbf{u}_k}(\mathbf{v}_k,\mu_k)$. It is a key step in our Algorithm \ref{alg:ANNI} to select an appropriate $\lambda_k$ and $\mu_k$. According to the convergence requirement, $\lambda_k$ and $\mu_k$ should be bounded, such that $\lambda_{\min}(J_{\mathbf{v}_k}(\mathbf{u}_k,\lambda_k))$, $\lambda_{\min}(J_{\mathbf{u}_{k+1}}(\mathbf{v}_k,\mu_k))\ge\gamma$ for some positive constant. Although we can estimate a suitable one satisfying $\tau_1\le\lambda_k,~\mu_k<\min\{\lambda_{\min}( \mathcal{A}_{\mathbf{v}}(\mathbf{u})),\lambda_{\min}( \mathcal{A}_{\mathbf{u}}(\mathbf{v}))\}$ for all unit $\mathbf{u}$ and $\mathbf{v}$, it may result in slow convergence in practice. 

In this subsection, we discuss the strategy of selecting $\lambda_k$ and $\mu_k$ for ANNI and present parts of $\lambda_k$ we have tried in Example \ref{exm:spin-1} with $U=[-16,16]^2$, $\beta_0=300$, $\beta_1=100$, $M=0$. Figure \ref{fig:tau} shows the results of ANNI with different strategies for choosing $\lambda_k$. We have compared three types of strategies. In Figure \ref{fig:tau}, `$\lambda_k=0$' represents that we use a fixed $\lambda_k=\tau_1$, and $\tau_1=-$1e1 (0, 5, etc.) represents that $\lambda_k=\max\left\{\tau_1,\min\left(\frac{\mathcal{A}_{\mathbf{v}_k}(\mathbf{u}_k)\mathbf{u}_k}{\mathbf{u}_k}\right)\right\}$. Then $\tau_1=$`-Inf' indicates that $\lambda_k=\min\left(\frac{\mathcal{A}_{\mathbf{v}_k}(\mathbf{u}_k)\mathbf{u}_k}{\mathbf{u}_k}\right)$ as used by Algorithm \ref{alg:NNI} (NNI). `st3' means that we set 
$$\lambda_k=\max\left\{\lambda_k,\min\left(\frac{\mathcal{A}_{\mathbf{v}_k}(\mathbf{u}_k)\mathbf{u}_k}{\mathbf{u}_k}\right)\right\}.$$ In general, the strategy setting $\lambda_k=\max\left\{\tau_1,\min\left(\frac{\mathcal{A}_{\mathbf{v}_k}(\mathbf{u}_k)\mathbf{u}_k}{\mathbf{u}_k}\right)\right\}$ will be more robust and efficient than a fixed $\lambda_k=\tau_1$. The larger $\tau_1$ is, the faster the algorithm converges while $J_{\mathbf{v}_k}(\mathbf{u}_k,\lambda_k)$ remains positive definite. On the other hand, the average iterations solving each linear system may also increase as $\tau_1$ increases. `st3' shows the best convergence rate and results in efficient solution of linear systems for most of our numerical experiments. In addition, for some cases set $\lambda_k=\mathbf{u}_k^T\mathcal{A}_{\mathbf{v}_k}(\mathbf{u}_k)\mathbf{u}_k$ when 
$\min\left(\frac{\mathcal{A}_{\mathbf{v}_k}(\mathbf{u}_k)\mathbf{u}_k}{\mathbf{u}_k}\right)<0$ may improve the results. 
\begin{figure}[htb]
	\centering
	\subfigure[nrmG versus the number of iterations]{
		\begin{minipage}[t]{0.6\linewidth}
			\centering
			\includegraphics[width=3in]{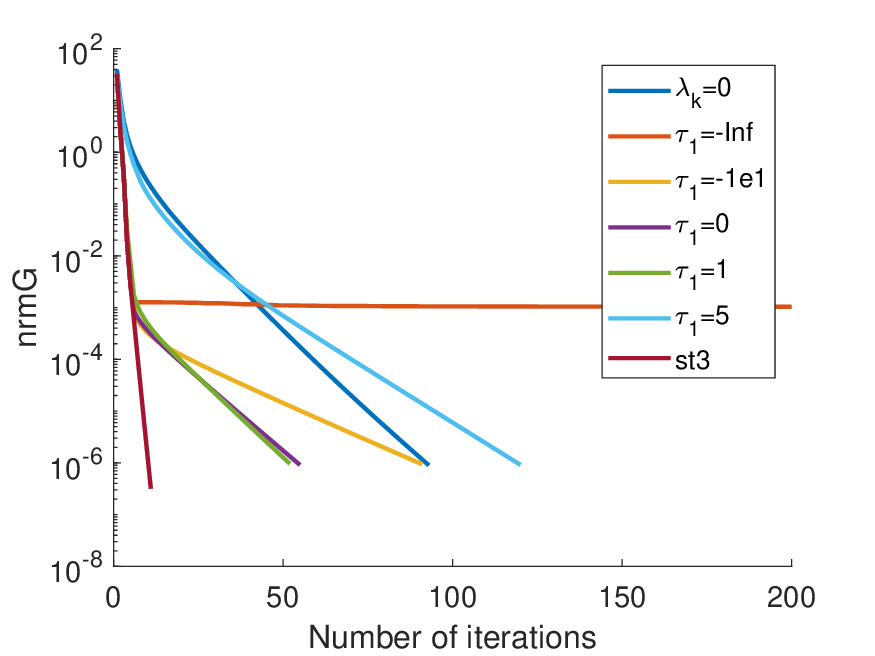}
			%\caption{fig1}
		\end{minipage}%
	}%
	\subfigure[objective value versus the number of iterations]{
		\begin{minipage}[t]{0.5\linewidth}
			\centering
			\includegraphics[width=3in]{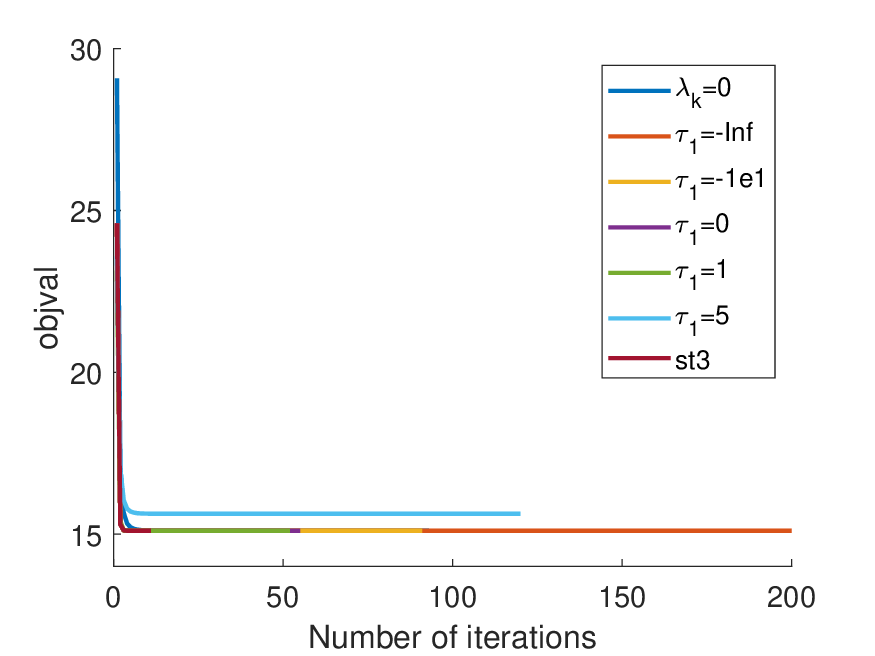}
			%\caption{fig2}
		\end{minipage}%
	}%
	\centering
	\caption{nrmG and objective value versus the number of iterations with different $\lambda_k$.}
	\label{fig:tau}
\end{figure}

%\begin{figure}[htb]
%	\centering
%	\subfigure[nrmG versus the number of iterations]{
%		\begin{minipage}[t]{0.6\linewidth}
%			\centering
%			\includegraphics[width=3in]{nrmG_change1_L16_beta300_M0_preLap.eps}
%			%\caption{fig1}
%		\end{minipage}%
%	}%
%	\subfigure[objective value versus the number of iterations]{
%		\begin{minipage}[t]{0.5\linewidth}
%			\centering
%			\includegraphics[width=3in]{fval_change1_L16_beta300_M0_preLap.eps}
%			%\caption{fig2}
%		\end{minipage}%
%	}%
%	\centering
%	\caption{When $\lambda_k=\max\left\{\tau_1,\min\left(\frac{\mathcal{A}_{\mathbf{v}_k}(\mathbf{u}_k)\mathbf{u}_k}{\mathbf{u}_k}\right)\right\}$, nrmG and objective value versus the number of iterations.}
%	\label{fig:tau}
%\end{figure}

It should be noted that the performance of ANNI applying for the Fourier pseudo-spectral scheme discretized problem is greatly influenced by the computation of \eqref{equ:computing}. See Figure \ref{fig:convergence}, ANNI seems to have the possibility to converge nearly quadratically, while in the worst case it may converge quite slowly, for the discretized BEC problem in 3D cases with a relatively large computational domain $U$. One important reason for this phenomena can be selecting an inappropriate $\lambda_k$ or $\mu_k$ and the inexact computation of \eqref{equ:computing} by CG. Better preconditioners or other efficient linear system solvers, and in particular a strategy for selecting better $\lambda_k,~\mu_k$ adaptively worth further study, which we would like to explore in our follow-up work.
\begin{figure}[htbp]
	\centering
	\subfigure[converge quadratically]{
		\begin{minipage}[t]{0.6\linewidth}
			\centering
			\includegraphics[width=3in]{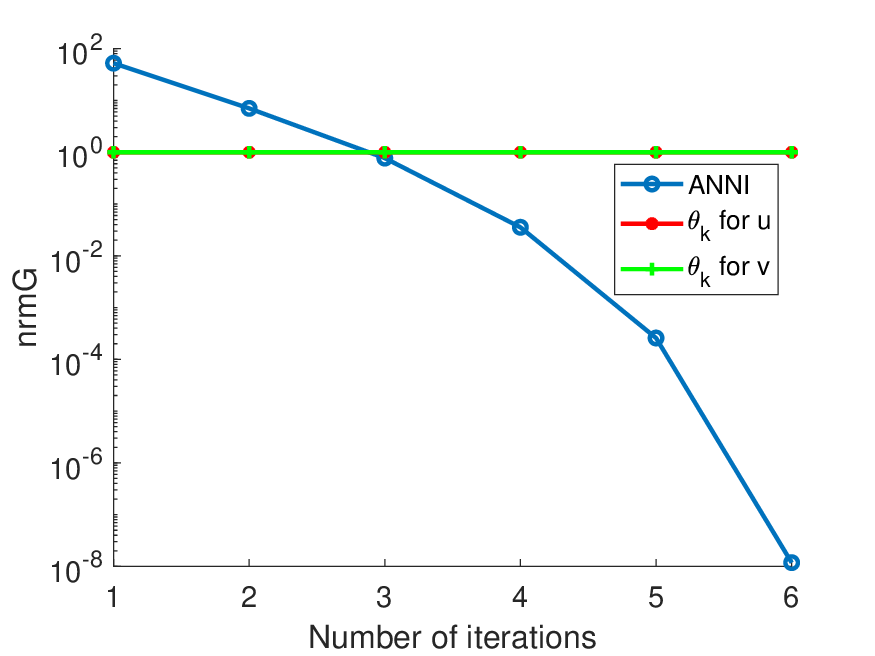}
			%\caption{fig1}
		\end{minipage}%
	}%
	\subfigure[converge rather slowly]{
		\begin{minipage}[t]{0.5\linewidth}
			\centering
			\includegraphics[width=3in]{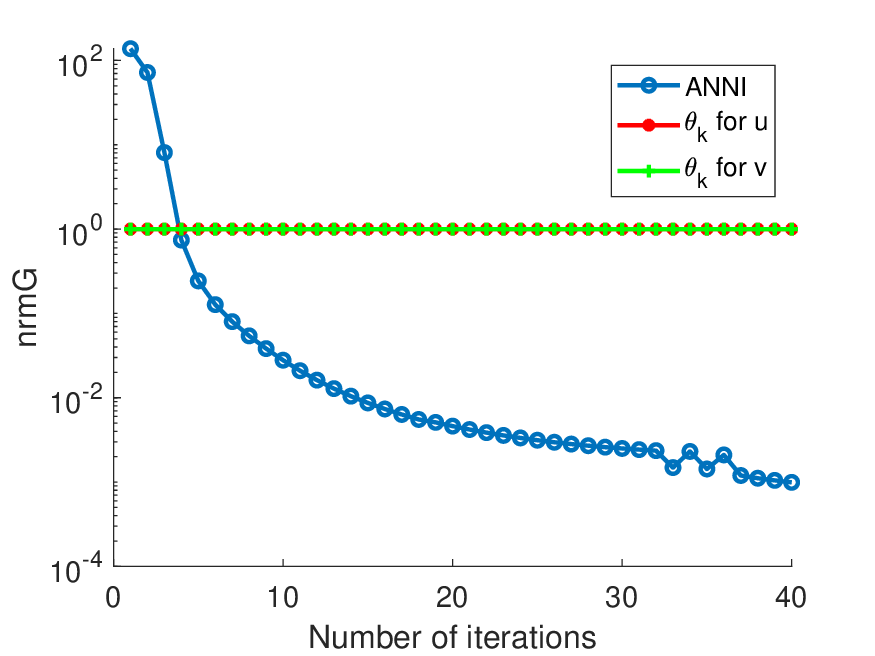}
			%\caption{fig2}
		\end{minipage}%
	}%
	\centering
	\caption{The convergence behaviour of ANNI. The left is for the 2D case of Example  \ref{exm:spin-1} with $U=[-16,16]^2,~\beta_0=0.3,~\beta_1=0.1,~M=0.9$. In this case, ANNI shows quadratically convergence. The right is for the 3D case of Example \ref{exm:spin-2} with $U=[-16,16]^3,~\beta_0=243,~\beta_1=12.1,~\beta_2=-13,~M=0.5$. In this case, ANNI converges rather slowly.}
	\label{fig:convergence}
\end{figure}

Nonetheless, we point out that regardless of the speed of convergence, the step size $\theta_k$ is almost always equal to one and hardly needs halving steps, as presented in Figure \ref{fig:convergence}. Whether it will eventually converge to 1 theoretically needs further analysis.

\subsection{Application in pseudo spin-1/2 BECs}
%Consider the specific formulation of the minimization problem for computing the ground states of spin-1/2 BEC without an internal atomic Josephson junction stated as follows \cite{bao2018mathematical}:
%\begin{equation}\label{equ:spin-1/2}
%	\begin{aligned}
	%		\min~& E(\Phi) = \int_{\mathbb{R}^d}\left[\frac{1}{2}(|\nabla \phi_1|^2+|\nabla \phi_2|^2)+V(\mathbf{x})(|\phi_1|^2+|\phi_2|^2)\right.\\
	%		&\left.+\frac{1}{2}\beta_{11}|\phi_1|^4+\frac{1}{2}\beta_{22}|\phi_2|^4+\beta_{12}|\phi_1|^2|\phi_2|^2\right]d\mathbf{x}\\
	%		\text{s.t.}~&\|\phi_1\|^2 = \alpha,~\|\phi_2\|^2=1-\alpha,
	%	\end{aligned}
%\end{equation}
%where $\alpha\in(0,1)$. $\mathbf{x}\in\mathbb{R}^d~(d=1,2,3)$ is the spatial coordinate vector, $\Phi=(\phi_1(\mathbf{x}), \phi_2(\mathbf{x}))^T$ is a two-component vector wave function. $V(\mathbf{x})$ is a real-valued external trapping potential. $\beta_{11},\beta_{22},\beta_{12}\in\mathbb{R}$ are interaction constants.

The two-component pseudo spin-1/2 BEC problem \eqref{equ:spin-1/2} can be transformed into \eqref{equ:problem} via different discretization schemes. As an example, we introduce the finite difference discretization of the energy functional \eqref{equ:spin-1/2} for the 1D case. Extensions to 2D and 3D are straightforward for tensor grids and details are omitted here for brevity. Due to the external trapping potential, the ground state of \eqref{equ:spin-1/2} decays exponentially as $|\mathbf{x}|\rightarrow\infty$. Thus we can truncate the energy functional from the whole space $\mathbb{R}^d$ to a bounded computational domain $U$ which is large enough such that the truncation error is negligible.

Let $U=[-L,L]$, $h=2L/n$ be the spatial mesh size and denote $x_j=-L+jh$ for $j=0,1,\cdots,n$. Let $\phi_{jl}$ be the numerical approximation of $\phi_l(x_j)$ for $j=0,1,\cdots,n$ and $l=1,2$ satisfying $\phi_{0l}=\phi_{nl}=0$, and denote $\mathbf{u}=(\sqrt{h/\alpha}\phi_{j1})$, $\mathbf{v}=(\sqrt{h/(1-\alpha)}\phi_{j2})\in\mathbb{R}^{n-1}~(j=1,\cdots,n-1)$. We have
\begin{equation*}
	\begin{aligned}
		E(\Phi)&\approx \sum_{j=0}^{n-1}\int_{x_j}^{x_{j+1}}\left[\sum_{l=1}^2\left(-\frac{1}{2}\phi_l(x)\phi_l^{''}(x)+V(x)\phi_l(x)^2\right)+\frac{\beta_{11}}{2}\phi_1(x)^4+\frac{\beta_{22}}{2}\phi_2(x)^4+\beta_{12}\phi_1(x)^2\phi_2(x)^2\right]dx\\
		&\approx h\sum_{j=1}^{n-1}\left[\sum_{l=1}^2\left(-\frac{1}{2}\phi_{jl}\frac{\phi_{j+1,l}-2\phi_{jl}+\phi_{j-1,l}}{h^2}+V(x_j)\phi_{jl}^2\right)+\frac{\beta_{11}}{2}\phi_{j1}^4++\frac{\beta_{22}}{2}\phi_{j2}^4+\beta_{12}\phi_{j1}^2\phi_{j2}^2\right]\\
		%		&= h\sum_{j=1}^{N-1}\frac{1}{2}(\frac{\phi_{j+1}-\phi_{j}}{h})^2+h\sum_{j=2}^{N-1}V(x_j)\phi_j^2+h\sum_{j=2}^{N-1}\frac{\beta}{2}\phi_j^4\\
		&=h\left[\frac{\alpha}{h}\mathbf{u}^TA\mathbf{u}+\frac{\beta_{11}}{2}\frac{\alpha^2}{h^2}\sum_{j=1}^{n-1}u_j^4
		+\frac{(1-\alpha)}{h}\mathbf{v}^TA\mathbf{v}+\frac{\beta_{22}}{2}\frac{(1-\alpha)^2}{h^2}\sum_{j=1}^{n-1}v_j^4+\beta_{12}\frac{\alpha(1-\alpha)}{h^2}\sum_{j=1}^{n-1}u_j^2v_j^2\right]\\
		&=\mathbf{u}^TA_1\mathbf{u}+\frac{\alpha^2\beta_{11}}{2h}\sum_{j=1}^{n-1}u_j^4+
		\mathbf{v}^TA_2\mathbf{v}+\frac{(1-\alpha)^2\beta_{22}}{2h}\sum_{j=1}^{n-1}v_j^4
		+\frac{\alpha(1-\alpha)\beta_{12}}{h}\sum_{j=1}^{n-1}u_j^2v_j^2,
	\end{aligned}
\end{equation*}
where $A_1=\alpha A$, $A_2=(1-\alpha)A$ and $A=(a_{jk})\in\mathbb{R}^{(n-1)\times(n-1)}$ given by
\begin{equation*}
	a_{jk}=
	\left\{
	\begin{array}{lrc}
		\frac{1}{h^2}+V(x_j), &j=k,\\
		-\frac{1}{2h^2},& |j-k|=1,\\
		0, &\text{otherwise}.
	\end{array}
	\right.
\end{equation*}
The constraints can be discretized as $\mathbf{u}^T\mathbf{u}=1$ and $\mathbf{v}^T\mathbf{v}=1$ similarly. Therefore $A_1$ and $A_2$ of \eqref{equ:problem} are irreducible nonsingular $M$-matrices via such discretization.

We compare the performance of ALM \eqref{equ:ALM} using Algorithm \ref{alg:NNI} (NNI) to solve subproblems (denoted by ALM$_{NNI}$), ANNI and BEFD \cite{bao2004ground} by testing on spin-1/2 BECs in 1D case stated at Example \ref{exm:1}. For BEFD , we set the step size $\Delta t$ to be $10$, and add another stopping rule as $\|\Phi_{k+1}-\Phi_k\|\le 10^{-7}$ or the iteration number $k\ge 1000$. 

In the subsequent tables, the columns ``f", ``nrmG" and ``CPU(s)" display the final objective function value, the final norm of the Riemannian gradient, and the total CPU time each algorithm spent to reach the stopping criterion. The column ``Iter" reports the number of iterations (the total numbers of inner iterations).

\begin{example}\label{exm:1}
	We consider the spin-1/2 BEC \cite{bao2004ground,bao2011ground}:
	\begin{itemize}
		\item 1D, $V(x)=\frac{1}{2}x^2+24\cos^2(x)$, $\beta_{11}:\beta_{12}:\beta_{22}=(1.03:0.97:1)\beta$ with $\beta>0$ in \eqref{equ:spin-1/2} , $U=[-16,16]$,  $n=2^{10}$. 
		%	\item 2D, $V(x,y)=\frac{1}{2}(x^2+y^2)+10[sin^2(\frac{\pi x}{2})+sin^2(\frac{\pi y}{2})]$, $\beta_{11}:\beta_{12}:\beta_{22}=(1:0.94:0.97)\beta$ with $\beta>0$, $U=[-16,16]^2$, $n=2^7$.
	\end{itemize}
\end{example}
The comparison results for ANNI, ALM$_{NNI}$ and BEFD are reported in Table \ref{tab:BEFD}. It does not come as a surprise that ANNI takes fewer time than ALM$_{NNI}$ since ALM \eqref{equ:ALM} conducts more iteration steps solving subproblems. Both ANNI and ALM \eqref{equ:ALM} show much higher efficiency than BEFD. An interesting phenomenon is that $$\hat{\lambda}_k=\min\left(\frac{\mathcal{A}_{\mathbf{v}_k}(\mathbf{u}_k)\mathbf{u}_k}{\mathbf{u}_k}\right)$$ might be still monotonically increasing as in NNI \cite{du2022newton}, see Figure \ref{fig:monotonic_1d}, where $\bar{\lambda}_k=\mathbf{u}_k^{T}\mathcal{A}_{\mathbf{v}_k}(\mathbf{u}_k)\mathbf{u}_k$ and $\hat{\mu}_k,~\bar{\mu}_k$ are defined similarly for $\mathbf{v}$. However, for problems with the pseudo-spectral discretization, the computed numerical $\hat{\lambda}_k$ will no longer have such a monotonous nature and even will not be positive anymore. The wave functions of the gound states computed by ANNI are given in Figure \ref{fig:spin-1/2}. 
%ANNI and ALM \eqref{equ:ALM} indeed find the positive solution as stated in Theorem \ref{thm:alm} and Corollary \ref{cor:finite}. The results of ANNI and ALM for the spin-1/2 BECs in 2D is presented in Table \ref{tab:2d}.

%\begin{example}\label{exm:2}
%	We take $d=1$, $V(x)=\frac{1}{2}x^2$; $d=2$, $V(x)=\frac{1}{2}(x^2+y^2)$; $d=3$, $V(x)=\frac{1}{2}(x^2+y^2+z^2)$. Let computational domain $\Omega=[-16,16],~[-16,16]\times[-16,16],~[-16,16]\times[-16,16]\times[-16,16]$, respectively with $n=2^7$.
%\end{example}
%Table \ref{tab:finite} summaries the results of ANNI and ALM for Example \ref{exm:2}.

% \usepackage{multirow}
\begin{table}[htb]
	\begin{center}
		\caption{Comparisons between ANNI, ALM$_{NNI}$ and BEFD for spin-1/2 BECs in 1D}
		\label{tab:BEFD}
		\setlength{\tabcolsep}{1mm}{
			\begin{tabular}{|cclccccclclcc|}
				\hline
				\multicolumn{1}{|c|}{\multirow{2}{*}{$\alpha$}} & \multicolumn{4}{c|}{ANNI}                             & \multicolumn{4}{c|}{ALM$_{NNI}$}                                   & \multicolumn{4}{c|}{BEFD}         \\ \cline{2-13} 
				\multicolumn{1}{|c|}{}                          & f       & nrmG  & Iter & \multicolumn{1}{c|}{CPU(s)} & f       & nrmG   & Iter     & \multicolumn{1}{l|}{CPU(s)}  & f       & nrmG   & Iter & CPU(s)  \\ \hline
				\multicolumn{13}{|c|}{$\beta=10$}                                                                                                                                                                        \\ \hline
				\multicolumn{1}{|c|}{0.2}                       & 6.8651  & 8.1e-8 & 7    & \multicolumn{1}{c|}{0.4089} & 6.8651  & 4.4e-7 & 5(23)   & \multicolumn{1}{l|}{0.6093}  & 6.8651  & 7.3e-7 & 34   & 1.3209  \\
				\multicolumn{1}{|c|}{0.5}                       & 6.8670  & 2.9e-7 & 7    & \multicolumn{1}{c|}{0.3708} & 6.8670  & 4.4e-7 & 6(27)   & \multicolumn{1}{l|}{0.7121}  & 6.8670  & 6.7e-7 & 34   & 1.3071  \\
				\multicolumn{1}{|c|}{0.8}                       & 6.9029  & 1.7e-7 & 6    & \multicolumn{1}{c|}{0.3824} & 6.9029  & 5.7e-7 & 5(22)   & \multicolumn{1}{l|}{0.5504}  & 6.9029  & 9.8e-7 & 34   & 1.3356  \\
				\multicolumn{1}{|c|}{0.9}                       & 6.9224  & 9.9e-8 & 6    & \multicolumn{1}{c|}{0.2928} & 6.9224  & 3.9e-8 & 5(20)   & \multicolumn{1}{l|}{0.5052}  & 6.9224  & 7.4e-7 & 35   & 1.3776  \\ \hline
				\multicolumn{13}{|c|}{$\beta=100$}                                                                                                                                                                       \\ \hline
				\multicolumn{1}{|c|}{0.2}                       & 17.1842 & 9.9e-7 & 87   & \multicolumn{1}{c|}{4.5023} & 17.1842 & 9.9e-7 & 108(323) & \multicolumn{1}{l|}{8.1111} & 17.1842 & 1.3e-6 & 739  & 28.5574 \\
				\multicolumn{1}{|c|}{0.5}                       & 17.1901 & 9.2e-7 & 134  & \multicolumn{1}{c|}{6.8237} & 17.1901 & 9.7e-7 & 129(368) & \multicolumn{1}{l|}{9.2166} & 17.1901 & 2.4e-6 & 437  & 16.9251 \\
				\multicolumn{1}{|c|}{0.8}                       & 17.3046 & 1.0e-6 & 95   & \multicolumn{1}{c|}{4.8613} & 17.3046 & 9.5e-7 & 87(279)  & \multicolumn{1}{l|}{6.9934} & 17.3046 & 4.6e-5 & 1000 & 38.9092 \\
				\multicolumn{1}{|c|}{0.9}                       & 17.3717 & 6.2e-7 & 29   & \multicolumn{1}{c|}{1.4786} & 17.3717 & 8.3e-7 & 22(87)  & \multicolumn{1}{l|}{2.1785}  & 17.3717 & 1.3e-6 & 1000  & 39.0564 \\ \hline
		\end{tabular}}
	\end{center}
\end{table}

\begin{figure}[htb]
	\centering
	\includegraphics[scale=0.6]{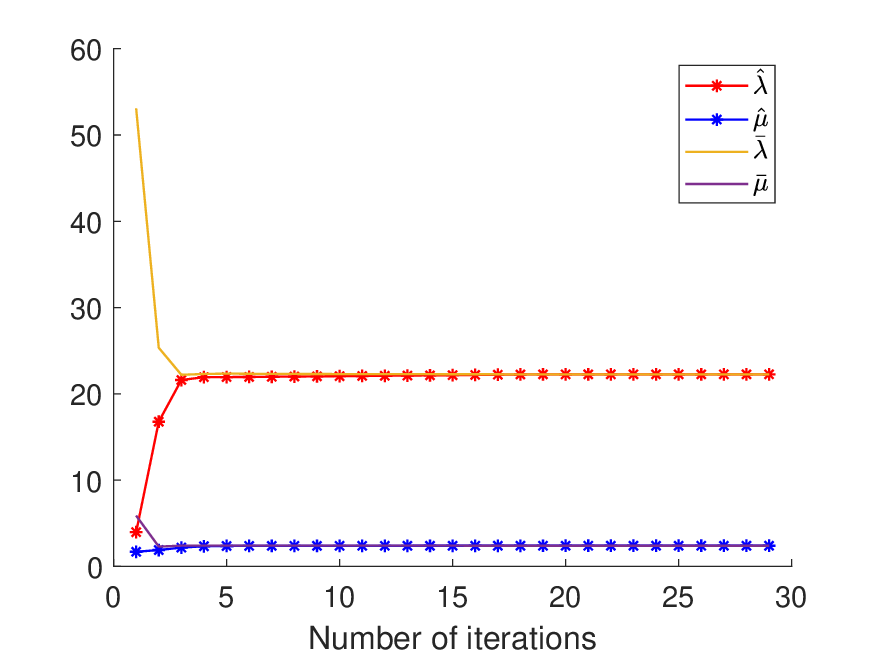}
	\caption{$\hat{\lambda}_k$, $\hat{\mu}_k$, $\bar{\lambda}_k$, $\bar{\mu}_k$ versus the number of iterations for Example \ref{exm:1} with $\beta=100$, $
		\alpha=0.9$.}
	\label{fig:monotonic_1d}
\end{figure}

\begin{figure}[htb]
	\centering
	\subfigure[$\alpha=0.2$]{
		\begin{minipage}[t]{0.6\linewidth}
			\centering
			\includegraphics[width=3in]{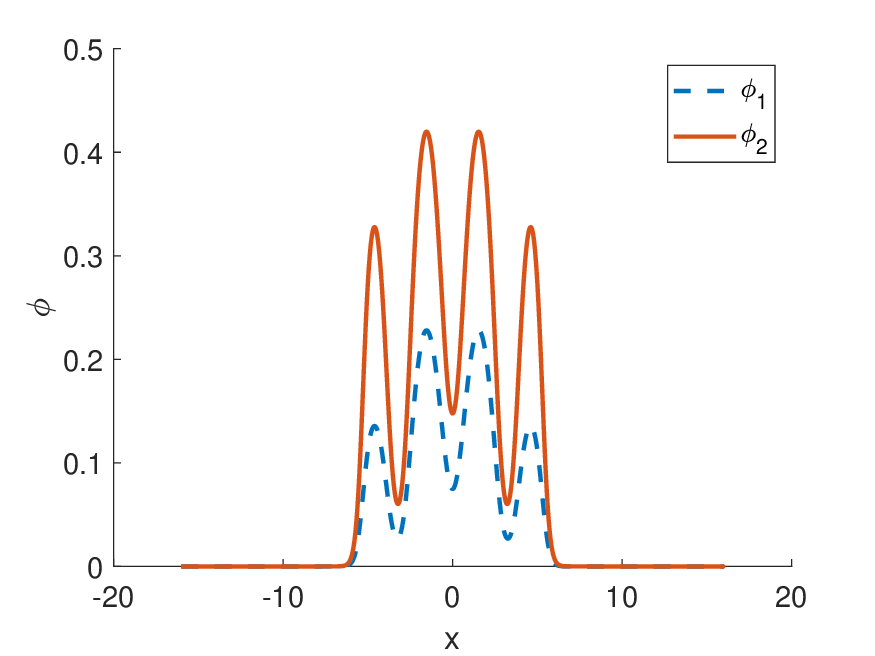}
			%\caption{fig1}
		\end{minipage}%
	}%
	\subfigure[$\alpha=0.9$]{
		\begin{minipage}[t]{0.5\linewidth}
			\centering
			\includegraphics[width=3in]{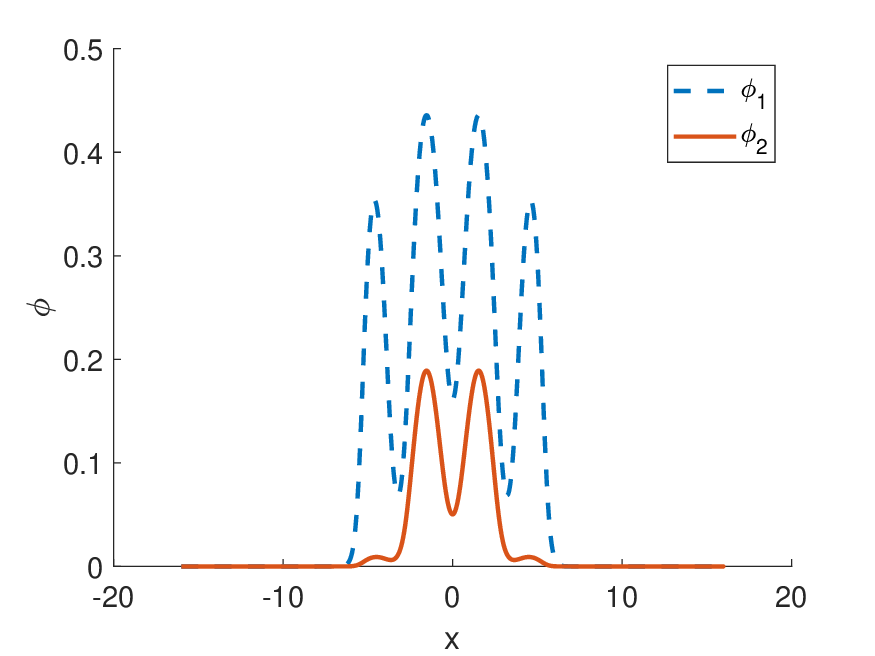}
			%\caption{fig2}
		\end{minipage}%
	}%
	\centering
	\caption{Wave functions of the ground state of a spin-1/2 BEC for Example \ref{exm:1} with $\beta=100$, i.e., $\phi_1(x)$ (dash line) and $\phi_2(x)$ (solid line). The left is for $\alpha=0.2$, and the right is for $\alpha=0.9$.}
	\label{fig:spin-1/2}
\end{figure}

\subsection{Application in anti-ferromagnetic spin-1 BECs}
In this subsection, we apply the ANNI and ALM with the Fourier pseudo-spectral scheme to compute the ground states of the spin-1 BEC in 2D and 3D cases with different interactions. We also compare them with the ARNT method \cite{tian2020ground}. It should be noted that ARNT is used to solve the discretized spin-$F$ BEC problems \eqref{equ:spin-1} and \eqref{equ:spin-2} directly instead of \eqref{equ:problem}. Hence, the initial discretized $\Phi_0=(\phi_{-F},\cdots,\phi_{F})~(F=1,2)$ of ARNT is chosen as $\frac{1}{\sqrt{n}}[1,\cdots,1]^T\left(\sqrt{\frac{1+M}{2}},0,\sqrt{\frac{1-M}{2}}\right)$ and $\frac{1}{\sqrt{n}}[1,\cdots,1]^T\left(\frac{\sqrt{2+M}}{2},0,0,0,\frac{\sqrt{2-M}}{2}\right)$ for \eqref{equ:spin-1} and \eqref{equ:spin-2}, respectively. The stopping criterion for ARNT then is set to be the norm of the Riemannian gradient $\|\text{grad}\tilde{E}(\mathbf{u}_k)\|_2\le 10^{-6}$ on the finest mesh; please refer to \cite{tian2020ground} for the formulation of $\tilde{E}(\mathbf{u})$. In ARNT we use all of the default parameters given in the article \cite{tian2020ground}. 
%We take the same stopping criterion for ARNT as \cite{tian2020ground}, that is, ARNT terminates when the norm of the Rimannian gradient smaller than $10^{-6}$.

The specific formulation of the minimization problem for computing the ground states of spin-1 BECs is stated as follows \cite{tian2020ground,bao2018mathematical}:
\begin{equation}\label{equ:spin-1}
	\begin{aligned}
		\min~& E(\Phi) = \int_{\mathbb{R}^d}\left\{\sum_{l=-1}^1\left(\frac{1}{2}|\nabla\phi_l|^2+(V(\mathbf{x})-pl+ql^2)|\phi_l|^2\right)+\frac{\beta_0}{2}|\Phi|^4+\frac{\beta_1}{2}|\mathbf{F}|^2\right\}d\mathbf{x}\\
		\text{s.t.}~&\int_{\mathbb{R}^d}\sum_{l=-1}^1|\phi_l(\mathbf{x})|^2d\mathbf{x}=1,~\int_{\mathbb{R}^d}\sum_{l=-1}^1l|\phi_l(\mathbf{x})|^2d\mathbf{x}=M,
	\end{aligned}
\end{equation}
where $M\in[-1,1]$ is the magnetization constraint, and $p,~q$ are the linear and quadratic Zeeman engergy shift. $\mathbf{F}:=\mathbf{F}(\Phi)=(\Phi^*f_x\Phi,\Phi^*f_y\Phi,\Phi^*f_z\Phi)^T$ is the spin vector where $f_x,f_y,f_z$ are Hermitian spin-1 matrices. $\beta_0$ and $\beta_1$ are density-dependent and spin-dependent interaction strength, respectively. According to \cite{bao2018mathematical,bao2013efficient}, when $q=0$, $M\in(-1,1)$, for the anti-ferromagnatic system $\beta_1>0$, \eqref{equ:spin-1} can be reduced to the two-component BEC problem as \eqref{equ:spin-1/2} with 
\[\beta_{11}=\beta_{22}=\beta_0+\beta_1,~\beta_{12}=\beta_0-\beta_1,~\alpha=\frac{1+M}{2}.\]
Hence, we can apply our algorithms to this kind of special anti-ferromagnetic spin-1 BECs. In all examples, we take $p=0$. Please refer \cite{tian2020ground,bao2018mathematical} for the detailed Fourier pseudo-spectral discretization procedure transforming \eqref{equ:spin-1/2} to the finite-dimensional optimization problem. In this way, $A_1$ and $A_2$ in \eqref{equ:problem} are positive definite Hermitian matrices instead of $M$-matrices. Hence, we use NRI \cite{huang2022newton} in ALM to solve subproblems (denoted by ALM$_{NRI}$).

We point out that although we replace $A_i~(i=1,2)$ with $(A_i+\bar{A_i})/2$ when $A_i$ is Hermitian in \eqref{equ:problem} for theoretical analysis, in practice, we still directly solve \eqref{equ:computing} with $A_i~(i=1,2)$ instead of $(A_i+\bar{A_i})/2$.

\begin{example}\label{exm:spin-1}
	The following cases are considered:
	\begin{itemize}
		\item 2D, $V(x,y)=\frac{1}{2}(x^2+y^2)+10[\sin^2(\frac{\pi x}{2})+\sin^2(\frac{\pi y}{2})]$, $n=2^9$. 
		\item 3D,  $V(x,y,z)=\frac{1}{2}(x^2+y^2+z^2)+100[\sin^2(\frac{\pi x}{2})+\sin^2(\frac{\pi y}{2})+\sin^2(\frac{\pi z}{2})]$, $n=2^7$.
	\end{itemize}
\end{example}

For Example \ref{exm:spin-1}, Table \ref{tab:spin-1} reports the results obtained by ANNI, ALM and ARNT with different $\beta_0$, $\beta_1$ and computational domains in 2D cases. Table \ref{tab:spin-1-3d} reports the results in 3D cases with a small computational domain. In general, ANNI can converge within around 12 iterations and shows the best performance taking the time cost and total iteration numbers into consideration. 

On the other hand, for a numerical solution to the continuous BEC problem, the computational domain in Table \ref{tab:spin-1-3d} for 3D cases with large interactions $\beta_0,~\beta_1$ may not be large enough. We present the results of 3D cases with a more reasonable large computational domain in Table \ref{tab:spin-1-3d-large}. In this case, the stopping criterion is set to be $\|\text{grad}f(\mathbf{u}_k,\mathbf{v}_k)\|_2\le 10^{-3}$ for ANNI and $\|\text{grad}\tilde{E}(\mathbf{u}_k)\|_2\le 10^{-3}$ for ARNT. ANNI may need to conduct more than one step modified NNI to obtain a better performance, while ALM$_{NRI}$ may stop with a larger function value. Therefore, we only present the results of ANNI and ARNT in Table \ref{tab:spin-1-3d-large}. The number in the bracket for ANNI then shows the total number conducting one step modified NNI. 

\begin{table}[htb]\small
	\caption{Numerical results of spin-1 BECs in 2D}
	\label{tab:spin-1}
	\begin{center}
		\setlength{\tabcolsep}{1mm}{
			\begin{tabular}{|ccccccccccccc|}
				\hline
				\multicolumn{1}{|c|}{}    & \multicolumn{4}{c|}{ANNI}                                  & \multicolumn{4}{c|}{ALM$_{NRI}$}                      & \multicolumn{4}{c|}{ARNT}         \\ \hline
				\multicolumn{1}{|c|}{$M$} & f       & nrmG   & Iter    & \multicolumn{1}{c|}{CPU(s)}   & f & nrmG & Iter & \multicolumn{1}{c|}{CPU(s)} & f       & nrmG   & Iter & CPU(s)  \\ \hline
				\multicolumn{13}{|c|}{$U=[-14,14]^2,~\beta_0=3,\beta_1=1$}                                                                                                                   \\ \hline
				\multicolumn{1}{|c|}{0.0} & 7.5123  & 3.9e-7 & 10  &  \multicolumn{1}{c|}{13.9275} & 7.5123& 2.5e-7& 10(224) &  \multicolumn{1}{c|}{31.5794} & 7.5123  & 9.9e-7 & 3(101)   & 30.4055 \\ \hline
				\multicolumn{1}{|c|}{0.2} & 7.5185  & 5.0e-7 & 10  & \multicolumn{1}{c|}{14.0306} & 7.5185  & 4.5e-7 & 9(215) & \multicolumn{1}{c|}{31.1417}  & 7.5185  &9.7e-7 & 3(203) & 38.7149 \\ \hline
				\multicolumn{1}{|c|}{0.5} & 7.5547  & 2.5e-7 & 9 & \multicolumn{1}{c|}{13.4479} & 7.5547 & 3.1e-7 & 7(193) & \multicolumn{1}{c|}{27.7404} & 7.5547  & 9.4e-7 & 3(202) &  38.3336 \\ \hline
				\multicolumn{1}{|c|}{0.9} & 7.6712  & 1.5e-8 & 7 & \multicolumn{1}{c|}{13.6484} & 7.6712 & 2.6e-7 & 5(166) & \multicolumn{1}{c|}{27.2636} & 7.6712  & 9.9e-7 & 3(105) & 29.9869  \\ \hline
				\multicolumn{13}{|c|}{$U=[-14,14]^2,~\beta_0=300,\beta_1=100$}                                                                                                               \\ \hline
				\multicolumn{1}{|c|}{0.0} & 15.1032 & 3.6e-7 & 11  & \multicolumn{1}{c|}{14.7718}  & 15.1032 & 9.5e-8 & 12(184) & \multicolumn{1}{c|}{18.4382}& 15.1032 & 9.8e-7 & 3(103)   & 22.4362  \\ \hline
				\multicolumn{1}{|c|}{0.2} & 15.1411 & 8.2e-7 & 11  & \multicolumn{1}{c|}{14.7785} & 15.1411 & 6.8e-7 & 13(180) & \multicolumn{1}{c|}{18.6051} & 15.1411 & 9.3e-7 & 3(107)  & 26.7064 \\ \hline
				\multicolumn{1}{|c|}{0.5} & 15.3436 & 8.8e-7 & 12 & \multicolumn{1}{c|}{18.0985} & 15.3436 & 8.4e-7 & 16(192) & \multicolumn{1}{c|}{25.0727} & 15.3436 & 9.4e-7 & 3(193) & 35.9765 \\ \hline
				\multicolumn{1}{|c|}{0.9} & 15.9621 & 7.5e-7 & 12 & \multicolumn{1}{c|}{21.5566} & 15.9621 & 7.7e-7 & 13(174) & \multicolumn{1}{c|}{29.0626} & 15.9621 & 1.0e-6 & 3(175) & 32.8107 \\ \hline
				\multicolumn{13}{|c|}{$U=[-16,16]^2,~\beta_0=0.3,\beta_1=0.1$}                                                                                                                 \\ \hline
				\multicolumn{1}{|c|}{0.0} & 6.5529 & 2.9e-7 & 6 & \multicolumn{1}{c|}{8.4577}  &  6.5529 & 3.9e-8 & 3(223)& \multicolumn{1}{c|}{30.8691}  & 6.5529 & 9.2e-7 & 3(65)  &17.9508 \\ \hline
				\multicolumn{1}{|c|}{0.2} & 6.5546 & 1.8e-7& 6  & \multicolumn{1}{c|}{8.6600}  & 6.5546  & 2.4e-8 &  3(221) & \multicolumn{1}{c|}{31.1892}  & 6.5546  & 9.9e-7 & 3(107) & 21.9349 \\ \hline
				\multicolumn{1}{|c|}{0.5} & 6.5639 & 3.8e-7 & 6 & \multicolumn{1}{c|}{8.9636}  &  6.5639 &  3.2e-8 & 3(332) & \multicolumn{1}{c|}{53.4762}  & 6.5639 & 9.2e-7 & 3(112)  & 21.8099 \\ \hline
				\multicolumn{1}{|c|}{0.9} & 6.5886  & 1.2e-8 & 6 & \multicolumn{1}{c|}{11.0125}  & 6.5886  & 4.0e-7 & 2(392)& \multicolumn{1}{c|}{67.6734}  & 6.5886 & 9.3e-7 & 3(92) & 21.1741\\ \hline
				\multicolumn{13}{|c|}{$U=[-16,16]^2,~\beta_0=3,\beta_1=1$}                                                                                                                   \\ \hline
				\multicolumn{1}{|c|}{0.0} & 7.5123  & 3.9e-7 & 10  &  \multicolumn{1}{c|}{13.9706} & 7.5123& 2.5e-7& 10(230) &  \multicolumn{1}{c|}{34.6489} & 7.5123  & 9.5e-7 & 3(97)   & 25.5942 \\ \hline
				\multicolumn{1}{|c|}{0.2} & 7.5185  & 5.0e-7 & 10  & \multicolumn{1}{c|}{14.1533} & 7.5185  & 4.6e-7 & 9(221) & \multicolumn{1}{c|}{34.3997}  & 7.5185  &9.9e-7 & 3(193) & 37.2005 \\ \hline
				\multicolumn{1}{|c|}{0.5} & 7.5547  & 1.1e-7 & 10 & \multicolumn{1}{c|}{14.7976} & 7.5547 & 1.7e-7 & 9(206) & \multicolumn{1}{c|}{33.2801} & 7.5547  & 8.8e-7 & 3(146) &  34.5899 \\ \hline
				\multicolumn{1}{|c|}{0.9} & 7.6712  & 1.4e-8 & 7 & \multicolumn{1}{c|}{12.9909} & 7.6712 & 3.3e-7 & 5(172) & \multicolumn{1}{c|}{32.1430} & 7.6712  & 8.5e-7 & 3(148) &  31.0586 \\ \hline
				\multicolumn{13}{|c|}{$U=[-16,16]^2,~\beta_0=300,\beta_1=100$}                                                                                                             \\ \hline
				\multicolumn{1}{|c|}{0.0} & 15.1032 & 3.2e-7 & 11   & \multicolumn{1}{c|}{14.2073}  &  15.1032 &9.6e-8 & 12(193) & \multicolumn{1}{c|}{19.6196}  & 15.1032 & 9.4e-7 & 3(95) & 19.2313 \\ \hline
				\multicolumn{1}{|c|}{0.2} & 15.1411 & 7.7e-7 & 11 & \multicolumn{1}{c|}{14.5451}  & 15.1411 & 5.7e-7 & 13(186) & \multicolumn{1}{c|}{19.4366}  & 15.1411 & 9.0e-7 & 3(134) & 25.5432 \\ \hline
				\multicolumn{1}{|c|}{0.5} & 15.3436 & 3.0e-7 & 13 & \multicolumn{1}{c|}{18.8408}  & 15.3436  & 9.6e-7 & 16(200) & \multicolumn{1}{c|}{26.2161}  & 15.3436 & 9.9e-7 & 3(174) & 33.1090\\ \hline
				\multicolumn{1}{|c|}{0.9} & 15.9621 & 8.1e-7 & 12 &  \multicolumn{1}{c|}{24.8223}  & 15.9621 & 6.7e-7 & 13(182) & \multicolumn{1}{c|}{30.2421}  & 15.9621 & 9.6e-7 & 3(136) & 28.1355 \\ \hline
				\multicolumn{13}{|c|}{$U=[-32,32]^2,~\beta_0=300,\beta_1=100$}                                                                                                               \\ \hline
				\multicolumn{1}{|c|}{0.0} & 15.1032 & 2.3e-7 & 11 & \multicolumn{1}{c|}{16.0256}  & 15.1032  & 1.1e-7 & 12(235) & \multicolumn{1}{c|}{23.8496}  & 15.1032 & 9.6e-7 & 3(77) & 17.2957 \\ \hline
				\multicolumn{1}{|c|}{0.2} & 15.1411 & 7.4e-7 & 11  & \multicolumn{1}{c|}{16.3504}  & 15.1411  &  5.7e-7 & 13(227) & \multicolumn{1}{c|}{23.9402}  & 15.1411 & 8.7e-7 & 3(85) & 20.9296\\ \hline
				\multicolumn{1}{|c|}{0.5} & 15.3436 & 4.2e-7 & 13 & \multicolumn{1}{c|}{21.8214}  &  15.3436 & 4.2e-7 &  16(238) & \multicolumn{1}{c|}{32.2173}  & 15.3436 & 9.8e-7 & 3(122) & 26.0992\\ \hline
				\multicolumn{1}{|c|}{0.9} & 15.9621 & 6.1e-7 & 12  & \multicolumn{1}{c|}{24.7720}  &  15.9621 &  6.7e-7  &  13(221) & \multicolumn{1}{c|}{45.3842} & 15.9621 & 9.9e-7 & 3(110) & 25.2033 \\ \hline
		\end{tabular}}
	\end{center}
\end{table}

\begin{table}[htb]
	\caption{Numerical results of spin-1 BECs in 3D with a small computational domain}
	\label{tab:spin-1-3d}
	\begin{center}
		\setlength{\tabcolsep}{1mm}{
			\begin{tabular}{|ccccccccccccc|}
				\hline
				\multicolumn{1}{|c|}{}    & \multicolumn{4}{c|}{ANNI}                               & \multicolumn{4}{c|}{ALM$_{NRI}$}                                    & \multicolumn{4}{c|}{ARNT}            \\ \hline
				\multicolumn{1}{|c|}{$M$} & f       & nrmG   & Iter & \multicolumn{1}{c|}{CPU(s)}   & f       & nrmG   & Iter     & \multicolumn{1}{c|}{CPU(s)}   & f       & nrmG   & Iter   & CPU(s)   \\ \hline
				\multicolumn{13}{|c|}{$U=[-2,2]^3,~\beta_0=3,~\beta_1=1$}                                                                                                                                                     \\ \hline
				\multicolumn{1}{|c|}{0.0} & 35.5665 & 5.7e-7 & 14   & \multicolumn{1}{c|}{125.4266} & 35.5665 & 4.9e-7 & 13(293)  & \multicolumn{1}{c|}{305.8099} & 35.5665 & 9.3e-7 & 3(141) & 320.2789 \\
				\multicolumn{1}{|c|}{0.2} & 35.5819 & 2.8e-7 & 15   & \multicolumn{1}{c|}{144.2015} & 35.5819 & 4.2e-7 & 11(275)  & \multicolumn{1}{c|}{286.1644} & 35.5819 & 8.9e-7 & 3(274) & 468.2085 \\
				\multicolumn{1}{|c|}{0.5} & 35.6633 & 2.0e-7 & 22   & \multicolumn{1}{c|}{258.7343} & 35.6633 & 3.0e-7 & 12(278)  & \multicolumn{1}{c|}{335.1609} & 35.6633 & 9.6e-7 & 3(333) & 520.1443 \\
				\multicolumn{1}{|c|}{0.9} & 35.9170 & 5.7e-7 & 11   & \multicolumn{1}{c|}{125.4266} & 35.9170 & 5.4e-7 & 8(244)   & \multicolumn{1}{c|}{410.6428} & 35.9170 & 8.8e-7 & 3(470) & 605.4597 \\ \hline
				\multicolumn{13}{|c|}{$U=[-2,2]^3,~\beta_0=300,~\beta_1=100$}                                                                                                                                                 \\ \hline
				\multicolumn{1}{|c|}{0.0} & 72.5380 & 6.1e-7 & 10    & \multicolumn{1}{c|}{34.1841}  & 72.5380 & 9.5e-7 & 10(212) & \multicolumn{1}{c|}{103.9954} & 72.5380 & 9.7e-7 & 3(52)  & 134.3032 \\
				\multicolumn{1}{|c|}{0.2} & 72.7579 & 4.4e-7 & 10    & \multicolumn{1}{c|}{50.3095}  & 72.7579 & 7.3e-7 & 11(215)  & \multicolumn{1}{c|}{104.7343}  & 72.7579 & 9.9e-7 & 3(50)  & 154.6070 \\
				\multicolumn{1}{|c|}{0.5} & 73.9404 & 4.5e-7 & 11    & \multicolumn{1}{c|}{59.2536}  & 73.9404 & 5.1e-7 & 11(216)  & \multicolumn{1}{c|}{110.3835} & 73.9404 & 9.5e-7 & 3(109)  & 181.7030 \\
				\multicolumn{1}{|c|}{0.9} & 77.4006 & 9.5e-7 & 14   & \multicolumn{1}{c|}{105.6792} & 77.4006 & 4.7e-7 & 13(222)  & \multicolumn{1}{c|}{162.9178} & 77.4006 & 9.4e-7 & 3(141) & 212.0984 \\ \hline
		\end{tabular}}
	\end{center}
\end{table}

\begin{table}[htb]
	\caption{Numerical results of spin-1 BECs in 3D with a larger computational domain}
	\label{tab:spin-1-3d-large}
	\begin{center}
	\begin{tabular}{|ccccccccc|}
		\hline
		\multicolumn{1}{|c|}{}    & \multicolumn{4}{c|}{ANNI}                                 & \multicolumn{4}{c|}{ARNT}            \\ \hline
		\multicolumn{1}{|c|}{M}   & f       & nrmG   & Iter   & \multicolumn{1}{c|}{CPU(s)}   & f       & nrmG   & Iter   & CPU(s)   \\ \hline
		\multicolumn{9}{|c|}{$U=[-16,16]^3,~\beta_0=300,~\beta_1=100$}                                                               \\ \hline
		\multicolumn{1}{|c|}{0.0} & 48.2473 & 9.6e-4 & 17(102)     & \multicolumn{1}{c|}{189.6522} & 48.2473 & 7.4e-4 & 2(78) & 148.8557 \\ \hline
		\multicolumn{1}{|c|}{0.2} & 48.3024 & 9.4e-4 & 28     & \multicolumn{1}{c|}{141.4743} & 48.3024 & 8.1e-4 & 4(168) & 248.1866 \\ \hline
		\multicolumn{1}{|c|}{0.5} & 48.6055 & 9.4e-4 & 31     & \multicolumn{1}{c|}{128.1350} & 48.6055 & 8.4e-4 & 2(38) & 143.9281 \\ \hline
		\multicolumn{1}{|c|}{0.9} & 49.5440 & 8.5e-4 & 12(69) & \multicolumn{1}{c|}{189.8731} & 49.5440 & 8.3e-4 & 2(90) & 202.8079 \\ \hline
	\end{tabular}
\end{center}
\end{table}

Although ANNI has the possibility to achieve a nearly quadratically convergence, it is not as efficient as expected when computing \eqref{equ:spin-1} in 3D cases over relatively large computational domains. For 2D cases, we have found a good parameter-selecting strategy `st3' that can ensure the efficiency of ANNI, while for 3D cases, such a strategy has not been figured out. As shown in Table \ref{tab:spin-1-3d-large}, the accuracy of nrmG for ANNI is low in this case. Possible explanations for this phenomenon may be that the main linear system involving $J_{\mathbf{v}}(\mathbf{u},\lambda)$ is computed by CG in an inexact iteration way, and different condition numbers of $J_{\mathbf{v}}(\mathbf{u},\lambda)$ affect the convergence rate significantly. It is worthwhile to find better preconditioners and linear system solvers, as well as better strategies to select $\lambda_k$ and $\mu_k$ adaptively as stated in subsection 6.1.2, which is beyond the scope of this article. We remark that for relatively weak interactions, for which a relatively small computational domain is enough in 3D cases, ANNI may still work well. 

\subsection{Application in spin-2 BECs}
Consider the ground state of the following spin-2 multi-component BEC problem \cite{tian2020ground,bao2018mathematical}:
\begin{equation}\label{equ:spin-2}
	\begin{aligned}
		\min~&E(\Phi) = \int_{\mathbb{R}^d}\left\{\sum_{l=-2}^2\left(\frac{1}{2}|\nabla\phi_l|^2+(V(\mathbf{x})-pl+ql^2)|\phi_l|^2\right)+\frac{\beta_0}{2}|\Phi|^4+\frac{\beta_1}{2}|\mathbf{F}|^2+\frac{\beta_2}{2}|A_{00}|^2\right\}d\mathbf{x}\\
		\rm{s.t.}~& \int_{\mathbb{R}^d}\sum_{l=-2}^2|\phi_l|^2d\mathbf{x}=1,\quad \int_{\mathbb{R}^d}\sum_{l=-2}^2l|\phi_l|^2d\mathbf{x}=M,
	\end{aligned}
\end{equation}
where $M\in[-2,2]$ is the magnetization constraint, $\beta_2$ is the spin-singlet interaction strength and all the other parameters $p,q,\beta_0$ and $\beta_1$ are the same as those in the spin-1 case, $\mathbf{F}$ is the spin vector defined by spin-2 matrices similar to $\mathbf{F}$ for spin-1 BECs. $A_{00}:=A_{00}(\Phi)=\Phi^TA\Phi$, where the matrix 
\[A=\frac{1}{\sqrt{5}}\begin{bmatrix}0&0&0&0&1\\0&0&0&-1&0\\0&0&1&0&0\\0&-1&0&0&0\\1&0&0&0&0\end{bmatrix}.\]

According to \cite{bao2018mathematical,bao2013efficient}, for the special case with $\beta_2<0$, $\beta_1>\frac{\beta_2}{20}$, \eqref{equ:spin-2} can be transformed to an equivalent two-component BEC problem \eqref{equ:spin-1/2} with 
\[\beta_{11}=\beta_{22}=\beta_0+4\beta_1,~ \beta_{12}=\beta_0-4\beta_1+\frac{2}{5}\beta_2,~ \alpha=\frac{2+M}{4}.\]
For the spin-2 BEC problem, we also discretize it by the Fourier pseudo-spectral scheme.

\begin{example}\label{exm:spin-2}
	The following cases are considered:
	\begin{itemize}
		\item 2D, $V(x,y)=\frac{1}{2}(x^2+y^2)+10[\sin^2(\frac{\pi x}{2})+\sin^2(\frac{\pi y}{2})]$, $n=2^8$. 
		\item 3D,  $V(x,y,z)=\frac{1}{2}(x^2+y^2+z^2)+100[\sin^2(\frac{\pi x}{2})+\sin^2(\frac{\pi y}{2})+\sin^2(\frac{\pi z}{2})]$, $n=2^7$.
	\end{itemize}
\end{example}

The stopping criterion is set to be the same as the spin-1 cases. Comparison results for spin-2 BECs are summarized in Tables \ref{tab:spin-2}, \ref{tab:spin-2-3d} and \ref{tab:spin-2-3d-large}. The numerical results are similar to those for the spin-1 cases.

\begin{table}[htb]
	\caption{Numerical results of spin-2 BECs in 2D}\label{tab:spin-2}
	\begin{center}
		\setlength{\tabcolsep}{1mm}{
			\begin{tabular}{|ccccccccccccc|}
				\hline
				\multicolumn{1}{|c|}{}    & \multicolumn{4}{c|}{ANNI}                                & \multicolumn{4}{c|}{ALM$_{NRI}$}                      & \multicolumn{4}{c|}{ARNT}         \\ \hline
				\multicolumn{1}{|c|}{$M$} & f       & nrmG   & Iter   & \multicolumn{1}{c|}{CPU(s)}  & f & nrmG & Iter & \multicolumn{1}{c|}{CPU(s)} & f       & nrmG   & Iter & CPU(s)  \\ \hline
				\multicolumn{13}{|c|}{$U=[-6,6]^2,~\beta_0=243,\beta_1=12.1,\beta_2=-13$}                                                                                                \\ \hline
				\multicolumn{1}{|c|}{0.0} & 14.3386 & 6.8e-7 & 15 & \multicolumn{1}{c|}{4.9019} & 14.3386 & 6.4e-7 & 17(354)   & \multicolumn{1}{c|}{10.8034} & 14.3386 & 9.6e-7 & 3(100)  & 11.6523 \\ \hline
				\multicolumn{1}{|c|}{0.5} & 14.3730 & 6.5e-7 & 16 & \multicolumn{1}{c|}{6.1160} & 14.3730  & 4.7e-7 & 17(324) & \multicolumn{1}{c|}{10.9607} & 14.3730 & 9.8e-7 & 3(252) & 20.1984 \\ \hline
				\multicolumn{1}{|c|}{1.5} & 14.6754 & 5.8e-7 & 16 & \multicolumn{1}{c|}{6.5847} & 14.6754  & 1.5e-7 & 16(293) & \multicolumn{1}{c|}{11.3848} & 14.6754 & 9.8e-7 & 3(283) & 21.4564 \\ \hline
				\multicolumn{13}{|c|}{$U=[-8,8]^2,~\beta_0=243,\beta_1=12.1,\beta_2=-13$}                                                                                                \\ \hline
				\multicolumn{1}{|c|}{0.0} & 14.3386 & 6.1e-7 & 15 & \multicolumn{1}{c|}{4.8523}  & 14.3386  & 6.4e-7 & 17(395) & \multicolumn{1}{c|}{13.4112}  & 14.3386 & 9.1e-7 & 3(86) & 9.0082 \\ \hline
				\multicolumn{1}{|c|}{0.5} & 14.3730 & 6.7e-7 & 16 & \multicolumn{1}{c|}{6.0414} & 14.3730  & 4.7e-7 & 17(342) & \multicolumn{1}{c|}{15.5215} & 14.3730 & 9.7e-7 & 3(188)   & 15.3355 \\ \hline
				\multicolumn{1}{|c|}{1.5} & 14.6754 & 8.0e-7 & 16 & \multicolumn{1}{c|}{6.4515} & 14.6754  & 1.5e-7 & 16(315) & \multicolumn{1}{c|}{16.6045} & 14.6754 & 9.6e-7 & 3(206) & 16.0236 \\ \hline
				\multicolumn{13}{|c|}{$U=[-8,8]^2,~\beta_0=5,\beta_1=1,\beta_2=-1$}                                                                                                      \\ \hline
				\multicolumn{1}{|c|}{0.0} & 7.8431  & 1.5e-7 & 7 & \multicolumn{1}{c|}{2.1721} & 7.8431  & 7.8e-7 & 3(177) & \multicolumn{1}{c|}{5.6536} & 7.8431  & 8.8e-7 & 3(115) & 12.3909 \\ \hline
				\multicolumn{1}{|c|}{0.5} & 7.8648  & 2.8e-7 & 7 & \multicolumn{1}{c|}{2.4744} &  7.8648 & 8.5e-7  & 3(175)  & \multicolumn{1}{c|}{8.2934} & 7.8648  & 9.8e-7 & 3(129) & 13.0496 \\ \hline
				\multicolumn{1}{|c|}{1.5} & 8.0695  & 2.6e-7 & 7 & \multicolumn{1}{c|}{2.6798} & 8.0695  & 1.6e-7 & 3(170)& \multicolumn{1}{c|}{8.6021} & 8.0695  & 9.6e-7 & 3(191) & 16.3471 \\ \hline
				\multicolumn{13}{|c|}{$U=[-8,8]^2,~\beta_0=500,\beta_1=100,\beta_2=-100$}                                                                                                \\ \hline
				\multicolumn{1}{|c|}{0.0} & 17.0099 & 9.6e-9 & 8 & \multicolumn{1}{c|}{2.1687} & 17.0099 & 2.8e-7 & 4(136) & \multicolumn{1}{c|}{3.6389}  & 17.0099 & 9.9e-7 & 3(107)   & 12.7421  \\ \hline
				\multicolumn{1}{|c|}{0.5} & 17.1973 & 3.1e-7 & 7 & \multicolumn{1}{c|}{2.0649} & 17.1973 & 1.6e-8 & 6(147) & \multicolumn{1}{c|}{5.5626} & 17.1973 & 8.3e-7 & 3(152) & 13.7900 \\ \hline
				\multicolumn{1}{|c|}{1.5} & 18.7462 & 1.3e-8 & 8 & \multicolumn{1}{c|}{2.8585} & 18.7462 & 8.6e-7 & 4(135) & \multicolumn{1}{c|}{5.5234}  & 18.7462 & 9.7e-7 & 3(139) & 14.3076 \\ \hline
		\end{tabular}}
	\end{center}
\end{table}

\begin{table}[htb]
	\caption{Numerical results of spin-2 BECs in 3D with a small computational domain}
	\label{tab:spin-2-3d}
	\begin{center}
		\setlength{\tabcolsep}{1mm}{
			\begin{tabular}{|ccccccccccccc|}
				\hline
				\multicolumn{1}{|c|}{}    & \multicolumn{4}{c|}{ANNI}                               & \multicolumn{4}{c|}{ALM$_{NRI}$}                                    & \multicolumn{4}{c|}{ARNT}            \\ \hline
				\multicolumn{1}{|c|}{$M$} & f       & nrmG   & Iter & \multicolumn{1}{c|}{CPU(s)}   & f       & nrmG   & Iter     & \multicolumn{1}{c|}{CPU(s)}   & f       & nrmG   & Iter   & CPU(s)   \\ \hline
				\multicolumn{13}{|c|}{$U=[-2,2]^3,\beta_0=243,\beta_1=12.1,\beta_2=-13$}                                                                 \\ \hline
				\multicolumn{1}{|c|}{0.0} & 68.1976 & 6.9e-7 & 15   & \multicolumn{1}{c|}{73.0746} &68.1976 &5.8e-7 & 22(292)  & \multicolumn{1}{c|}{120.7202} & 68.1976 & 9.7e-7 & 3(52) & 221.4341 \\
				\multicolumn{1}{|c|}{0.5} & 68.4124 & 9.9e-7 & 16   & \multicolumn{1}{c|}{82.9188} & 68.4124 & 2.8e-7 & 18(280)  & \multicolumn{1}{c|}{118.7541} & 68.4124 & 7.9e-7 & 3(123) & 307.8404\\
				\multicolumn{1}{|c|}{1.5} & 70.2096 & 4.9e-7 & 18 & \multicolumn{1}{c|}{122.9136} & 70.2096 & 8.0e-7 & 18(276)   & \multicolumn{1}{c|}{151.8907} & 70.2096 & 9.2e-7 & 3(168) & 361.8182 \\ \hline
				\multicolumn{13}{|c|}{$U=[-4,4]^3,\beta_0=5,\beta_1=1,\beta_2=-1$}                                                                                                                                                 \\ \hline
				\multicolumn{1}{|c|}{0.0} & 35.8034 & 6.4e-7 & 9    & \multicolumn{1}{c|}{93.9882}  & 35.8034 & 1.4e-7 & 5(202) & \multicolumn{1}{c|}{192.6018} & 35.8034 & 9.6e-7 & 3(64)  & 294.5964 \\
				\multicolumn{1}{|c|}{0.5} & 35.8422 & 6.3e-7 & 10   & \multicolumn{1}{c|}{110.3853}  & 35.8422 & 1.2e-7 & 5(200)  & \multicolumn{1}{c|}{209.0550} & 35.8422 & 1.0e-6 & 3(144)  & 420.2329 \\
				\multicolumn{1}{|c|}{1.5} & 36.2132 & 1.7e-7 & 12  & \multicolumn{1}{c|}{200.5335} & 36.2132 & 3.8e-9 & 5(207) & \multicolumn{1}{c|}{325.2515} & 36.2132 & 9.2e-7 & 3(144) & 373.1969 \\ \hline
		\end{tabular}}
	\end{center}
\end{table}

\begin{table}[htbp]
	\caption{Numerical results of spin-2 BECs in 3D with a larger computational domain}
	\label{tab:spin-2-3d-large}
	\begin{center}
	\begin{tabular}{|ccccccccc|}
		\hline
		\multicolumn{1}{|c|}{}    & \multicolumn{4}{c|}{ANNI}                               & \multicolumn{4}{c|}{ARNT}            \\ \hline
		\multicolumn{1}{|c|}{M}   & f       & nrmG   & Iter & \multicolumn{1}{c|}{CPU(s)}   & f       & nrmG   & Iter   & CPU(s)   \\ \hline
		\multicolumn{9}{|c|}{$U=[-16,16]^3,~\beta_0=243,~\beta_1=12.1,~\beta_2=-13$}                                               \\ \hline
		\multicolumn{1}{|c|}{0.0} & 47.0361 & 9.6e-4 & 26   & \multicolumn{1}{c|}{220.5565} & 47.0361 & 9.9e-4 & 2(45) & 231.1504 \\ \hline
		\multicolumn{1}{|c|}{0.5} & 47.0888 & 9.9e-4 & 40   & \multicolumn{1}{c|}{169.2251} & 47.0888 & 9.5e-4 & 5(266) & 641.8977 \\ \hline
		\multicolumn{1}{|c|}{1.5} & 47.5668 & 9.6e-4 & 25   & \multicolumn{1}{c|}{215.7031} & 47.5668 & 9.4e-4 & 4(167) & 462.2247 \\ \hline
	\end{tabular}
\end{center}
\end{table}

\begin{remark}
	Finally, we want to remark that for the numerical experiments in this paper, the cascadic multigrid is not used in ANNI, and the initial data is $\mathbf{u}_0=\frac{1}{\sqrt{n}}[1,\cdots,1]^T$, $\mathbf{v}_0=\frac{1}{\sqrt{n}}[1,\cdots,1]^T\in\mathbb{R}^n$. Under this setting, the strategy `st3' for $\lambda_k$ and $\mu_k$ shows the best performance. However, if other initial data like the discretized $\phi_0(\textbf{x})=\frac{1}{\pi^{d/4}}e^{-(x_1^2+\cdots+x_d^2)/2}$, or the multigrid technique is used, the strategy for selecting $\lambda_k,~\mu_k$ may require appropriate modifications. We will not present the details of their implementation within this article for brevity.
\end{remark}

\section{Conclusion}\label{sec:conclude}
In this paper, the discretized energy functional minimization problem of a class of special multi-component BECs was considered. The original problem was reduced to a structured nonconvex optimization problem over spherical constraints. We apply the alternating minimization scheme to solve it. In particular, an easy-to-implement alternating Newton-Noda iteration (ANNI) with the one-step modified NNI in each inner iteration was designed to solve the discretized problem. The global convergence is guaranteed based on the sufficient descent property of the objective value. We also proved that the proposed algorithms are positivity preserving under mild conditions. Furthermore, ANNI can be applied to a class of more general multi-block nonconvex minimization problems.

Numerical results on different multi-component BECs are provided to confirm our theoretical results and demonstrate the efficiency of ANNI, especially for the weak interaction cases, for which a relatively small computational domain is enough. In this case, ANNI shows the possibility to converge nearly quadratically. However, it converges slowly when the computational domain is rather large in 3D cases. The explicit convergence rate analysis concerning different $J_{\mathbf{v}}(\mathbf{u},\lambda)$ needs further research, including whether the local quadratic convergence can be theoretically guaranteed under certain conditions. It is still an interesting problem to study further whether to improve the performance of ANNI via better preconditioners, other efficient linear system solvers, or approximating $J_{\mathbf{v}}(\mathbf{u},\lambda)$ by well-conditioned matrices. In particular, it is worth exploring the strategy for computing a better $\lambda_k$ or $\mu_k$ adaptively to improve ANNI in future work.

\section*{Acknowledgments} The authors are grateful to Dr. Yuanzhou Fang for his valuable suggestions about this article and Dr. Tonghua Tian for providing the code of ARNT kindly. The authors would also like to thank the anonymous referees for their valuable suggestions, which help us to improve the paper greatly.

\bibliographystyle{siamplain}
\bibliography{ref}
\end{document}